\pdfoutput=1
\documentclass[11pt]{article}

\usepackage{graphicx}
\usepackage{xcolor} % A package to add color.
\usepackage{tensor}
\usepackage{fullpage} % Sets all margins to 1 in.  
\usepackage{verbatim}
\usepackage{mathrsfs}

\usepackage{mathtools}
\usepackage{amssymb,bbm,amsmath,amsfonts,amsthm}
\usepackage{vmargin} 
\usepackage[linktocpage,colorlinks=true,raiselinks]{hyperref}
\hypersetup{urlcolor=black, citecolor=red, linkcolor=blue}
\usepackage{cleveref}
\usepackage{tkz-euclide}
\usepackage{multirow}
\usepackage{pgfplots}
\pgfplotsset{compat=newest}
\usepgfplotslibrary{fillbetween}
\usepackage{capt-of}

\usepackage{geometry}
\numberwithin{equation}{section}

\usepackage{scalerel,stackengine}
\stackMath
\newcommand\wwhat[1]{%
\savestack{\tmpbox}{\stretchto{%
  \scaleto{%
    \scalerel*[\widthof{\ensuremath{#1}}]{\kern-.6pt\bigwedge\kern-.6pt}%
    {\rule[-\textheight/2]{1ex}{\textheight}}%WIDTH-LIMITED BIG WEDGE
  }{\textheight}% 
}{0.5ex}}%
\stackon[1pt]{#1}{\tmpbox}%
}
\usepackage{ulem}

\usepackage{enumitem}
\setlist[enumerate]{leftmargin=1.5em}
\setlist[itemize]{leftmargin=1.5em}

\usepackage{seqsplit}

\usepackage{appendix}

\definecolor{green}{rgb}{0,0.8,0} % Redefines the color green.
\newtheorem{theorem}{Theorem}[section]
\newtheorem{corollary}[theorem]{Corollary}
\newtheorem{lemma}[theorem]{Lemma}
\newtheorem{proposition}[theorem]{Proposition}

\theoremstyle{definition}

\newtheorem{condition}[theorem]{Condition}

\theoremstyle{remark}
\newtheorem{remark}[theorem]{Remark}
\newtheorem*{Notation}{Notation and convention}

\numberwithin{equation}{section}
\newcommand{\RN}[1]{%
  \textup{\uppercase\expandafter{\romannumeral#1}}%
}

\newcommand{\nnrm}[1]{{\vert\kern-0.25ex\vert\kern-0.25ex\vert #1 
    \vert\kern-0.25ex\vert\kern-0.25ex\vert}}

\newcommand{\ud}{\mathrm{d}}

\newcommand{\dt}{\mathrm{d} t}

\newcommand{\bbP}{\mathbb P}

\newcommand{\bbR}{\mathbb R}

\newcommand{\calE}{\mathcal E}

\newenvironment{AMS}{}{}
\newenvironment{keywords}{}{}
\title{Convergence and non-convergence phenomena in Euler-Maxwell to MHD transitions
}

\author{Dong-ha Kim\thanks{Chung-Ang University. E-mail address: \url{kimdongha91@cau.ac.kr}}
	\and
	Junha Kim\thanks{Ajou University.  E-mail address: \url{junha02@ajou.ac.kr}
		}
	\and
	Jihoon Lee\thanks{
		Chung-Ang University. E-mail address: \url{jhleepde@cau.ac.kr}
	}
    }
\begin{document}
\maketitle

\begin{abstract}
In this work, we investigate the difference estimate for a class of  Euler-Maxwell system and those of magnetohydrodynamics (in short, MHD) systems in three dimensions. We decompose the Euler-Maxwell system into three parts, namely the MHD system, \textit{auxiliary linear system} and \textit{error part system}. As a result, we obtain the convergence of the velocity of the fluid $u$, electric fields $E$ and magnetic fields $B$ from the Euler-Maxwell system toward the MHD system in $L^p_tL^2_x$ as the speed of light $c$ approaches infinity for $p\in[1,\infty]$. We also derived non-convergence results of electric current $j$ or $cE$, and these results are classified by a certain threshold for $p$. Finally, we investigate how the $L^2$-energy flow of Euler-Maxwell system evolves as c tends to infinity, leading to the vanishing of Amp\`{e}re’s equation in the Euler-Maxwell system.
\end{abstract}



\tableofcontents
\allowdisplaybreaks % allowing long equations

%
%---------------------
\section{Introduction}
%---------------------
%
\subsection{Motivation of the work}
We consider incompressible Euler-Maxwell system with Ohm's law in three-dimensional space:
\begin{equation}\label{main_eqn}
\begin{cases}
 \begin{aligned}
	\text{(Euler equation)}\quad &	\partial_t u + (u \cdot \nabla)u + \nabla p = j \times B, \quad &\operatorname{div} u =0,\\
	\text{(Amp\`{e}re’s equation)}\quad &	\frac {1}{c} \partial_t E - \nabla \times B = -j,  \quad & \operatorname{div} E = 0,\\
	\text{(Faraday’s equation)}\quad &	\frac{1}{c}\partial_t B + \nabla \times E = 0, \qquad & \operatorname{div} B =  0,\\
	\text{(Ohm's law)}\quad &	j=\sigma(cE + \mathbb{P}(u \times B)), \qquad & \operatorname{div} j = 0.
	\end{aligned}
\end{cases}	
\end{equation}
%subject to initial data 
%\begin{equation*}\label{initial_nsm}
%u\vert_{t=0} =u^c_0,\qquad E\vert_{t=0} = E^c_0 ,\qquad B\vert_{t=0} = B^c_0.
%\end{equation*}
In the above system, $t \in {\mathbb{R}}_{+}$ and $x\in {\mathbb{R}}^3$ denote time and space variables and the operator $\mathbb{P}=I+\nabla(-\Delta)^{-1} \operatorname{div}$ denotes Leray projection. The field $u=(u_1,u_2,u_3)=u(t,x)$ denotes the velocity of the fluid, while $E=(E_1,E_2,E_3)=E(t,x)$ and $B=(B_1,B_2,B_3)=B(t,x)$ stand for the electric fields and magnetic fields, respectively. Moreover, 
$p=p(t,x)$ denotes the scalar pressure which is also an unknown. Additionally, $j$ denotes the electric current which is not an unknown since it is completely determined by $u,E$ and $B$ by Ohm's law. Finally, the positive constants $c \geq c_0$ and $\sigma>0$ represent the speed of light and the electrical conductivity, for some fixed constant $c_0>0$. Since we are not interested in the inviscid limit of $\sigma$, we fix $\sigma=1$ in this paper.

To ensure compatibility of the initial data, we assume the divergence free condition $\operatorname{div}u_0 = \operatorname{div}E_0=\operatorname{div}B_0=0.$ One can observe that the divergence-free condition $\operatorname{div}E=\operatorname{div}B=\operatorname{div}j=0$ is preserved over time by taking the divergence of Maxwell's equations in \eqref{main_eqn} which consist of Amp\`{e}re's and Faraday's equations. Moreover, the Euler-Maxwell system \eqref{main_eqn} satisfies the formal energy conservations in the following form:
\begin{equation}\label{energy-cons2}
\frac12 \frac{\ud}{\ud t}(\| u \|_{L^2}^2 +\| E\|_{L^2}^2+\| B \|_{L^2}^2 )+ \| j \|_{L^2}^2 =0.
\end{equation}

The Euler-Maxwell system \eqref{main_eqn} describes the dynamics of a charged fluid, known as plasma, subject to Lorentz force $j\times B$. A plasma is often referred to the fourth state of matter, which describes a wide variety of macroscopically neutral substances containing many interacting free electrons and ionized atoms. 
Fluids like plasma are known to possess the properties of perfect conductors; hence, it may be physically assumed that the term $\frac{1}{c} \partial_t E$ vanishes. Consequently, \eqref{main_eqn} can be formally reduced to the usual incompressible magnetohydrodynamic (in short, MHD) system:
\begin{equation} \label{eq-MHD}
	\left\{
	\begin{aligned}
		&\partial_t u +(u \cdot \nabla)u +\nabla p = (\nabla \times B)\times B,\qquad & \operatorname{div} u =0,\\
		& \partial_t B -\Delta B =\nabla \times (u \times B), \qquad &\operatorname{div} B =0,
	\end{aligned}	
	\right.
\end{equation}
which describes the behavior of a conducting inviscid fluid or a plasma in the presence of a magnetic field. There have been extensive mathematical studies on a rigorous treatment to MHD system. 
%------------- MHD literature ----------------------%
For these, we refer to references \cite{duvaut1972inequations,sermange1983some,kozono1989weak}.

In the last decade, the Navier-Stokes-Maxwell system (i.e., Euler-Maxwell \eqref{main_eqn} with a dissipation term of $u$) has attracted a lot of attention. 
The global existence of two-dimensional solutions has been studied in \cite{masmoudi2010global,ibrahim2011global,germain2014well}, and the uniform bound of solutions was established in  \cite{arsenio2020solutions}.
In the three-dimensional case, \cite{arsenio2020solutions} obtained the global existence of weak solutions and \cite{arsenio2024axisymmetric} proved the global existence of axisymmetric solutions with strong singular limit result toward MHD system.
See also other related works \cite{arsenio2024global,arsenio2019vlasov,ibrahim2018time,ibrahim2012local,10.21099/tkbjm/1496160397,10.3792/pjaa.62.181,nunez2005existence}.

Compared to the extensive studies on the Navier-Stokes-Maxwell system, relatively little research has been done on the incompressible Euler-Maxwell system connecting MHD system. We refer to recent results \cite{arsenio2022damped} and \cite{arsenio2024stability} for the first global regularity results on solutions to the Euler-Maxwell system with two-dimensional normal structure.
The authors in \cite{arsenio2022damped} and \cite{arsenio2024stability} adopt Yudovich's approach to construct a unique global solution of the Euler-Maxwell system, where the electromagnetic field has Sobolev regularity $H^s(\mathbb{R}^2)$ for $s \in (\frac{7}{4},2)$. Moreover, they show that an appropriate norms of solutions are uniformly bounded with respect to $c$ and established the strong convergence of the solutions. To the best of our knowledge, in the literature there has been no results on general three-dimensional Euler-Maxwell system. This paper focuses on the asymptotic limit for the smooth solutions to the three-dimensional Euler-Maxwell system \eqref{main_eqn} as $c \to \infty$.

Before stating the main results, we first review some key observations.
Let $(u^c,B^c,E^c)$ and $(\overline u,\overline B)$ satisfy the Euler-Maxwell system and those of the MHD system, respectively, with $(u^c,E^c,B^c)\vert_{t=0}=(u_0,E_0,B_0)$ and $(\overline u,\overline B)\vert_{t=0}=(u_0,B_0)$. Formally, as $c \to \infty$, one may speculate the solutions to the Euler-Maxwell system converge \textit{in certain sense} as follows:
\begin{equation}\label{eq1657wed}
\begin{aligned}
(u^c,E^c,B^c) &\longrightarrow (\overline u,0,\overline B), \\
j^c &\longrightarrow \nabla \times \overline  B, \\
cE^c &\longrightarrow \nabla \times \overline  B - \mathbb{P}(\overline u\times \overline B). \\
\end{aligned}
\end{equation}
However, such convergences may not be strong enough to guarantee the exact convergence of energy equality: Given that the initial data for the Euler-Maxwell system is $(u_0,E_0,B_0)$, we set the initial data for the MHD system as $(u_0,B_0)$.
On the other hands, the energy conservation law from \eqref{energy-cons2} yields that
\begin{equation}
\label{eng_1}\begin{gathered}
        \frac{1}{2} \left(\| u^c(t) \|_{L^2}^2 + \| E^c(t) \|_{L^2}^2 + \| B^c(t) \|_{L^2}^2\right) + \int_0^t\|j^c(\tau)\|_{L^2}^2 \,\ud \tau \\ = \frac{1}{2} \left(\| u_0 \|_{L^2}^2 + \| E_0 \|_{L^2}^2 + \| B_0 \|_{L^2}^2\right)
        \end{gathered}
    \end{equation} for all $c \geq c_0$.
The energy conservation of MHD system (see \eqref{L2_est_MHD}) gives \begin{equation}\label{eng_2}
\begin{gathered}
        \frac{1}{2} \left(\| \overline{u}(t) \|_{L^2}^2 + \| \overline{B}(t) \|_{L^2}^2\right) +\int_0^t\|\nabla \times \overline{B}(\tau)\|_{L^2}^2 \,\ud \tau = \frac{1}{2} \left(\| u_0 \|_{L^2}^2 + \| B_0 \|_{L^2}^2\right).
\end{gathered}
    \end{equation}
We combine \eqref{eng_1}-\eqref{eng_2} to see 
\begin{equation}\label{energyjump0305}
\begin{gathered}
    \frac{1}{2} \left(\| u^c(t) \|_{L^2}^2 + \| E^c(t) \|_{L^2}^2 + \| B^c(t) \|_{L^2}^2\right) + \int_0^t\|j^c(\tau)\|_{L^2}^2 \,\ud \tau \\= \frac{1}{2} \left(\| \overline{u}(t) \|_{L^2}^2 + \| \overline{B}(t) \|_{L^2}^2\right) + \int_0^t\|\nabla \times \overline{B}(\tau)\|_{L^2}^2 \,\ud \tau + \frac{1}{2}\| E_0 \|_{L^2}^2
\end{gathered}
\end{equation}
    
We assume that the solutions $(u^c,B^c,E^c)$ to the Euler-Maxwell systems satisfy the convergences \eqref{eq1657wed} in a sufficiently strong sense, such that
\begin{equation}\label{eq1658wed}
\begin{gathered}
\left(\|u^c(t)\|_{L^2},\|E^c(t)\|_{L^2},\|B^c(t)\|_{L^2}\right) \rightarrow \left(\|\overline{u}(t)\|_{L^2},0,\|\overline{B}(t)\|_{L^2}\right), \\
 \int_0^t\|j^c(\tau)\|_{L^2}^2 \,\ud \tau \rightarrow \int_0^t\|\nabla \times \overline{B}(\tau)\|_{L^2}^2 \,\ud \tau,
\end{gathered}
\end{equation}
for $t >0$. However, \eqref{energyjump0305} and \eqref{eq1658wed}  lead to
\begin{equation*}
    \lim_{c\to\infty}\|E_0\|_{L^2}=0,
\end{equation*}
which is not true in general.

This observation naturally suggests that, assuming $\|E_0\|_{L^2} \not = 0$, the solutions to the Euler-Maxwell system may not necessarily converge in certain topologies, raising the question of which topologies allow for convergence and which do not.
Furthermore, we formally have shown, using a proof by contradiction, that the total convergence \eqref{eq1658wed} cannot occur. To gain a constructive understanding, one may ask how the energy difference between the Euler-Maxwell system and the MHD system is redistributed as $c$ tends to infinity. See Corollary~\ref{cor_energyjump} for the details.

We summarize the main contributions of our results.
\begin{itemize}
    \item \textbf{Convergence of the solutions.} We establish the convergence of the solution $(u^c,E^c,B^c)$ to the Euler-Maxwell system in $L^p_{t,\text{loc}}L^2_x$ for $p \in [1,\infty]$. We also obtain the convergence rate.
    \item \textbf{Thresholds for non-convergence in $L^p_tL^2_x$.} Additionally, we precisely identify the thresholds for the non-convergence of $j^c$ and $cE^c$ in $L^p_{t,\text{loc}}L^2_x$, varying $p \in [1,\infty]$. We not only find the rate of convergence, but also compute the blow-up rate in case where $j^c$ and $cE^c$ fail to  converge.
    \item \textbf{Flow of the initial $L^2$-energy.} We investigate  where the initial $L^2$-energy of $E_0^c$ flows. Specifically, as $c\to \infty$, this energy flows instantaneously into the $L^2(0,t;L^2)$-norm of $j^c$. Also, the $L^2$-norm of $E^c(t)$ converges to zero for each time $t>0$.
\end{itemize}

\subsection{Main results}
%\begin{Notation} We use the following notation in this section.
%\begin{itemize}
%\item For any positive $A$ and $B$, we use the notation $A\lesssim B$ to mean that there exists a generic constant $C > 0$ that can depend on $m$, $c_0$, $M$, $\overline{M}$, and $T$ such that $A \leq CB$. Moreover, $A \simeq B$ is denoted when $A \lesssim B$ and $B \lesssim A$.
%\end{itemize}
%\end{Notation}

In this paper, we shall study the Cauchy problem of \eqref{main_eqn}. We begin by introducing our initial data setup.
\begin{condition}[Initial data condition]\label{cond_initial} A positive integer $m$ satisfies $m>5/2$. For each $c \geq c_0$, we denote initial data $(u^c,E^c,B^c)\vert_{t=0} =(u^c_0,E^c_0,B^c_0) \in H^m(\mathbb{R}^3)$ with $\operatorname{div}u_0^c = \operatorname{div}E_0^c=\operatorname{div}B_0^c=0$. We assume that initial data $(u^c_0,B^c_0)$ converges in $L^2$ strong topology, as $c \to \infty$, towards some $(u_0, B_0) \in H^m(\mathbb{R}^3)$.
For the initial data $E^c_0$, we assume the following uniform boundedness
\begin{equation*}
    \sup_{c\geq c_0}\|E^c_0\|_{H^1}<+\infty.
\end{equation*}
\end{condition}
%The intention of such initial data condition is for an analysis that encompasses all such cases \eqref{main_eqn} and \eqref{NSM-2010} at once.
%In the case of $\gamma=1$, it is easy to see that the Cauchy problem \eqref{main_eqn} with \eqref{initial_data_gamma} is equivalent to that of equation \eqref{NSM-2010} (with $\nu=0$) subject to initial data to $(u_0,\Psi_0,B_0)$ as the model adopted in \cite{arsenio2015derivation}, \cite{arsenio2020solutions}. In the case of $\gamma=0$, the Cauchy poblem \eqref{main_eqn}-\eqref{initial_data_gamma} is exactly the system \eqref{main_eqn} subject to initial data $(u_0,\Psi_0,B_0)$.

Before stating our main theorem, we remark the following: our main theorem proceeds by assuming the existence of strong solutions  \eqref{main_eqn} and \eqref{eq-MHD} on arbitrary given time interval $[0,T]$, then estimating the difference between the two solutions. In section \ref{sec2}, we will justify this assumption  by demonstrating the local-in-time existence of the Euler-Maxwell system \eqref{main_eqn} for some small $T^*>0$ with uniform estimate in $L^{\infty}([0,T^*];H^m(\mathbb{R}^3))$ (with respect to $c\geq c_0$).
We use the following notation throughout this work:
\begin{equation}\label{def_Ebar}
    \overline{E}:=\bbP(\nabla \times \overline{B}-\overline{u} \times \overline{B}),
\end{equation}
and
\begin{equation*}
   \overline{j}:= \nabla \times \overline{B}.
\end{equation*}
Our main theorem may now be stated as follows.
\begin{theorem} [Difference estimates] \label{thm_conv}
We assume that there exists $T>0$ satisfying the following: For each $c \geq c_0$, let $(u^c,B^c,E^c) \in C([0,T];H^m(\bbR^3))$ be the solution of Euler-Maxwell system \eqref{main_eqn} satisfying Condition \ref{cond_initial}.
By picking $(u_0,B_0)$ up in Condition \ref{cond_initial}, let $(\overline{u},\overline{B}) \in C([0,T];H^m(\bbR^3))$ be the solution of the MHD system \eqref{eq-MHD} subject to initial data $(u_0, B_0)$.
We also assume that there exist two constants $M>0$ and $\overline{M}>0$ which do not depend on $c \geq c_0$, satisfying the following upper bounds
\begin{equation}\label{ass_eng_EM}
    \begin{aligned}
		\int_0^T \| B^c(t) \|_{L^{\infty}}^2 \,\ud t + \int_0^T \| j^c(t) \|_{L^{\infty}} \,\ud t \leq M,
		\end{aligned}
	\end{equation}
and
\begin{equation}\label{ass_eng_MHD2016}
		\begin{aligned}
			\sup_{t \in [0,T]} \left( \|\overline{u}(t) \|_{H^m}  + \| \overline{B}(t) \|_{H^m} \right)  + \left(\int_0^{T} \|\nabla \overline{B}(t)\|_{H^m}^2 \,\ud t\right)^{\frac{1}{2}} \leq \overline{M}.
		\end{aligned}
	\end{equation}
Below, recalling $\overline{E}$ in \eqref{def_Ebar}, we set
\begin{equation}\label{Eczero_0214}
    \calE^c_0 := \|u^c_0-u_0\|_{L^2}+\|B^c_0-B_0\|_{L^2} + \frac{1}{c^2} \left(\| cE^c_0 - \overline{E}(0) \|_{H^1} +1\right).
\end{equation}  
Then, for any $p \in [1,+\infty]$, there exists a positive constant $C$\footnote{The constant $C$ depends only on $T$, $M$, and $\overline{M}$.} such that
\begin{equation}\label{c_rate1}
	\begin{aligned}
		\| (u^c,\, B^c) - (\overline{u},\,\overline{B}) \|_{L^{p}(0,T ; L^2(\bbR^3))} \leq C \, \calE^c_0,
	\end{aligned}
\end{equation}
and
\begin{equation}\label{c_rate2}
		\| cE^c - (\overline{E} + e^{-c^2t} (cE^c_0-\overline{E}(0))) \|_{L^p(0,T;L^2)} \leq C (1+ c^{1-\frac{2}{p}})\,\calE^c_0
\end{equation}
for all $c \geq c_0$. If we further assume that 
\begin{equation}\label{ass_eng_EM_2}
%    \sup_{t \in [0,T]}\|B^c(t)\|_{L^{\infty}}\leq M,
    \|B^c(t)\|_{L^{p}(0,T ; L^{\infty}(\bbR^3))}\leq M,
\end{equation}
then 
\begin{equation}\label{c_rate3}
		\| \,j^c - (\overline{j} + e^{-c^2t} (cE^c_0-\overline{E}(0))) \|_{L^p(0,T;L^2)} \leq C (1+c^{1-\frac{2}{p}})\,\calE^c_0
\end{equation}
holds for all $c \geq c_0$. \end{theorem}

\begin{remark}
The estimate \eqref{c_rate3} remains valid for $p \in [1,2]$, without assuming \eqref{ass_eng_EM_2} but only \eqref{ass_eng_EM}.
\end{remark}
Theorem~\ref{thm_conv} naturally yields various points, such as the convergence or non-convergence of solutions. The convergence of $(u^c,B^c) \to (\overline{u},\overline{B})$ and $E^c \to 0$ are deduced from \eqref{c_rate1} and \eqref{c_rate2}, respectively. 
On the other hand, from \eqref{c_rate2}, one can expect that the convergence or non-convergence of $(cE^c,j^c)$ largely depends on the value of $\|e^{-c^2t}\|_{L^p_t}\|cE^c_0-\overline{E}(0)\|_{L^2_x}$.
We now state this in Corollary \ref{cor_conv_uB}-\ref{cor_conv_nonconv}.

\begin{corollary}[Convergence of solutions $u^c,B^c$ and $E^c$]\label{cor_conv_uB}
Assume all assumptions in Theorem~\ref{thm_conv}. Namely, $(u^c,B^c,E^c)$ is a solution of Euler-Maxwell system \eqref{main_eqn}, and $(\overline{u},\overline{B})$ is a solution of MHD system \eqref{eq-MHD} with the initial data referred in Theorem~\ref{thm_conv}. Then, $(u^c,B^c)$ converges to $(\overline{u},\,\overline{B})$ in $L^{\infty}(0,T;L^2(\mathbb{R}^3))$.i.e. 
\begin{equation}\label{limit_ub}
    \lim_{c \to \infty}\|(u^c -\overline{u}, B^c - \overline{B})\|_{L^{\infty}(0,T;L^2(\mathbb{R}^3))}=0.
\end{equation}
Moreover, the electric field $E^c$ converges to zero in $L^p(0,T;L^2(\mathbb{R}^3))$ for every $p\in[1,\infty)$. In other words,
\begin{equation}\label{limit_E}
    \lim_{c \to \infty}\|E^c\|_{L^{p}(0,T;L^2(\mathbb{R}^3))}=0 \quad \text{for every } p \in [1,\infty).
\end{equation}
For $p=\infty$, the electric field $E^c$ converges to zero in $L^{\infty}(0,T;L^2(\mathbb{R}^3))$ if $\|E^c_0\|_{L^2(\mathbb{R}^3)}$ converges to zero as $c \to \infty$. Conversely, if $\|E^c_0\|_{L^2(\mathbb{R}^3)}$ does not tend to zero, then $E^c$ fails to converges to any limit in $L^{\infty}(0,T;L^2(\mathbb{R}^3))$.
\end{corollary}

\begin{corollary}[Convergence and Non-convergence of $cE^c$ and $j^c$]\label{cor_conv_nonconv} Assume all assumptions in Theorem~\ref{thm_conv}. 
Then, we can classify the convergence or non-convergence of $\{cE^c\}_{c\geq c_0}$ and $\{j^c\}_{c\geq c_0}$ in $L^p(0,T;L^2(\mathbb{R}^3))$ as $p$ varies, as follows:
\begin{enumerate}[label=(\roman*)]
\item $1 \leq p < 2$: It holds that
\begin{equation*}
    \lim_{c \to \infty}\|(cE^c -\overline{E}, j^c - \overline{j})\|_{L^{p}(0,T;L^2(\mathbb{R}^3))}=0 \quad \text{for every } p \in [1,2).
\end{equation*}
\item $p=2$: It holds that
\begin{equation*}
 \lim_{c\to \infty}\|E^c_0\|_{L^2(\mathbb{R}^3)}=0 \quad \text{implies} \quad\lim_{c \to \infty}\|(cE^c -\overline{E}, j^c - \overline{j})\|_{L^{2}(0,T;L^2(\mathbb{R}^3))}=0.
\end{equation*}
Conversely, if $\|E^c_0\|_{L^2(\mathbb{R}^3)}$ does not tend to zero, then neither $\{cE^c\}_{c\geq c_0}$ nor $\{j^c\}_{c\geq c_0}$ converges to any limit in $L^{2}(0,T;L^2(\mathbb{R}^3))$ as $c \to \infty$.
\item $2<p<\infty$: Suppose that initial data decays in the sense that
\begin{equation}\label{eq0131mon}
    \lim_{c\to \infty} c^{1-\frac{2}{p}}\left(\|E^c_0\|_{L^2}+\|u^c_0-u_0\|_{L^2}+\|B^c_0-B_0\|_{L^2}\right)=0.
\end{equation}
Then, $(cE^c,j^c)$ converges to $(\overline{E}, \overline{j})$ in $L^p(0,T;L^2(\mathbb{R}^3))$.

On the other hands, assume that the decay rate has the lower bound in the sense that
\begin{equation*}
    \limsup_{c\to \infty} c^{1-\frac{2}{p}}\left(\|E^c_0\|_{L^2} -\|u^c_0-u_0\|_{L^2}-\|B^c_0-B_0\|_{L^2}\right) >0.
\end{equation*}
Then, neither $\{cE^c\}_{c\geq c_0}$ nor $\{j^c\}_{c\geq c_0}$ converges to any limit in $L^p(0,T;L^2(\mathbb{R}^3))$ as $c \to \infty$.
\end{enumerate}
\end{corollary}
\begin{remark}
The case (iii) in Corollary \ref{cor_conv_nonconv} can be understood via the following special case. To see this, we temporarily take $u^c_0=u_0$ and $B^c_0=B_0$ so that
\begin{equation*}
    \calE^c_0=\frac{1}{c^2}(\|cE^c_0-\overline{E}(0)\|_{H^1}+1) = \frac{1}{c}\|E^c_0\|_{H^1} + O\left(\frac{1}{c^2}\right).
\end{equation*}
If we assume that $\|E^c_0\|_{L^2}$ decays with certain rates with
\begin{equation*}
        \|E^c_0\|_{L^2(\mathbb{R}^3)} \lesssim \frac{1}{c^{1-\beta}} \quad \text{for some }\beta \in [0,1),
\end{equation*}
then, $cE^c \to \overline{E}$ and $j^c \to \overline{j}$ in $L^p_tL^2_x$ for all $p \in [1,\frac{2}{\beta})$. Moreover, by carefully seeing the estimate in the proof, we can also obtain the convergence rate explicitly as
\begin{equation*}
    \|cE^c-\overline{E}\|_{L^p(0,T:L^2(\mathbb{R}^3))}\lesssim \frac{1}{c^{\frac{2}{p}-\beta}} \quad\text{and}\quad 
    \|j^c-\overline{j}\|_{L^p(0,T:L^2(\mathbb{R}^3))}\lesssim \frac{1}{c^{\frac{2}{p}-\beta}}.
\end{equation*}
If we assume that the lower bound of the decay rate of $\|E^c_0\|_{L^2}$ by
\begin{equation*}
    \|E^c_0\|_{L^2(\mathbb{R}^3)} \gtrsim \frac{1}{c^{1-\gamma}} \quad \text{for some }\gamma \in (0,1],
\end{equation*}
then, neither $cE^c$ nor $j^c$ converges to any limit in $L^p(0,T:L^2(\mathbb{R}^3))$ for every $p \in [\frac{2}{\gamma},\infty]$. Furthermore, we compute the blow-up rate by
\begin{equation*}
    \|cE^c-\overline{E}\|_{L^p(0,T:L^2(\mathbb{R}^3))}\gtrsim c^{\gamma-\frac{2}{p}}\quad\text{and}\quad 
    \|j^c-\overline{j}\|_{L^p(0,T:L^2(\mathbb{R}^3))}\gtrsim c^{\gamma-\frac{2}{p}}.
\end{equation*}
We omit the details.
\end{remark}

We employ Condition~\ref{cond_initial} and the convergence of $(u^c,B^c)$ established in Corollary~\ref{cor_conv_uB}. Then, some minor modification of the estimate \eqref{energyjump0305} leads to
\begin{equation}\label{eq1721tue}
       \lim_{c\to\infty}  \left(\frac{1}{2}\| E^c(t) \|_{L^2}^2 - \frac{1}{2}\| E^c_0 \|_{L^2}^2 + \int_0^t\|j^c(\tau)\|_{L^2}^2 \,\ud \tau \right) = \int_0^t\|\overline{j}\|_{L^2}^2 \,\ud \tau ,
\end{equation}
for all $t \in(0,T]$. Notice that \eqref{eq1721tue} essentially comes from the simple energy conservation identity.
We now analyze how the energy contribution from $\frac{1}{2}\|E^c_0\|_{L^2}^2$, compared to the MHD system, is distributed. 
We state our finial corollary concerning the energy flow.
\begin{corollary}\label{cor_energyjump}
Assume all assumptions in Theorem~\ref{thm_conv}. 
Then, we have
\begin{equation}\label{energy_jump}
        \lim_{c \to \infty} \|  E^c(t) \|_{L^2}^2 = 0,\quad \mbox{and}\quad \lim_{c\to \infty} \left(\int_0^t\|j^c(\tau)\|^2_{L^2}\, \ud \tau -\frac{1}{2}\|E^c_0\|_{L^2}^2\right)=  \int_0^t\|\,\overline{j}\,(\tau) \|_{L^2}^2\, \ud \tau ,
    \end{equation}
for each fixed $t \in (0,T]$.
\end{corollary}

\subsection{Strategy of the proof}
To prove Theorem~\ref{thm_conv}, we proceed with the following steps. We assume that the solutions $(\overline{u}$, $\overline{B})$ to MHD system  are given. First, we introduce \textit{auxiliary linear system} by
\begin{equation}\label{linear_1706}
	\begin{cases}
	    \begin{aligned}
		&\frac{1}{c}\partial_t E_{L} - \nabla \times B_{L} + cE_{L} = \overline{E},\quad &\operatorname{div} E_L =0, \\
		&\frac{1}{c}\partial_t B_{L} + c\nabla \times E_{L} = \nabla \times \overline{E}, \quad &\operatorname{div} B_L = 0,\\
        &\qquad E_{L}\vert_{t=0}=E^c_0,\quad B\vert_{t=0}=0.
	\end{aligned}
	\end{cases}
\end{equation}
A significant portion of our analysis depends on the convergence characteristics of the auxiliary linear system \eqref{linear_1706}.  
Finally, we denote \textit{the error part} $(\widetilde{u}, \widetilde{E}, \widetilde{B})$ by
\begin{equation}\label{Error_term}
    \begin{cases}
        \widetilde{u} := u^c - \overline{u}, \\
        \widetilde{E} := E^c - E_L, \\
        \widetilde{B} := B^c-\overline{B}-B_L.
    \end{cases}
\end{equation}
Here, we did not specify all the dependencies on $c$ in the solution for readability. Then, the error part $(\widetilde{u}, \widetilde{E}, \widetilde{B})$ solves the system
\begin{equation*}\label{Error_eqn}
\begin{cases}
	    \begin{aligned}
		\partial_t \widetilde{u} + (\widetilde{u} \cdot \nabla)\widetilde{u} + \nabla \widetilde{p} &= j^c \times B^c - \overline{j} \times \overline{B} - (\overline{u} \cdot \nabla)\widetilde{u} - (\widetilde{u} \cdot \nabla)\overline{u}, &\operatorname{div} \widetilde{u}=0,\\
		\frac{1}{c}\partial_t \widetilde{E} -  \nabla \times \widetilde{B} +  c\widetilde{E} &= - \bbP (u^c \times B^c  - \overline{u} \times \overline{B}  ), &\operatorname{div} \widetilde{B}=0, \\
		\frac{1}{c}\partial_t \widetilde{B} + \nabla \times \widetilde{E} &= 0, &  \operatorname{div} \widetilde{E} = 0,\\
         \widetilde{u}\vert_{t=0}=u^c_0-u_0,\quad &\widetilde{E}\vert_{t=0}=0,\quad \widetilde{B}\vert_{t=0}= B^c_0-B_0.
	\end{aligned}
	\end{cases}
\end{equation*}
We shall crucially use the fact that initial data of error part tend to zero where it holds due to Condition~\ref{cond_initial}, to show that the error part converges to zero as $c \to \infty$ with certain specified decay rate, and this is also important step in our analysis. Notice that the system \eqref{linear_1706} appears similar to the damped Maxwell
system introduced in \cite{arsenio2015derivation}, but it exhibits subtle differences in the source term.  Here, the design of the auxiliary linear system \eqref{linear_1706} is based on the formal assumption that $cE^c$ converges to $\overline{E}$ and $B^c$ converges to $\overline{B}$ as well as the expectation that $(E_L,B_L)$ closely resembles $(E^c,B^c-\overline{B})$. In other words, it is designed by inversely predicting the form of the source term under the assumption that $cE_L$ converges to $\overline{E}$ and $B_L$ converges to $0$ as $c \to \infty$.

In summary,  we decompose the solution to \eqref{main_eqn} into three pieces, i.e.,
	\[
	(\mbox{soln. to \eqref{main_eqn}})= ({\mbox{soln. to limit system \eqref{eq-MHD}}}) + ({\mbox{soln. to auxiliary linear system}})+({\mbox{error part}}).
	\]
Since the error part always tends to zero, the difference estimate between the MHD system and Euler-Maxwell systems can be understood by analyzing the linearized system. This is the overall story of the main theorem.

%\begin{remark}
%    {\color{red}
%    Moreover, for any given $\varepsilon>0$, there exists a constant $C>0$ not depending on $c$ such that
%\begin{equation}
%	\begin{aligned}
%		\| (u^c,\, E^c,\, B^c) - (u,\,E,\,B) \|_{L^\infty([\varepsilon,T] ; L^2(\bbR^3))} \leq Cc^{-2}.
%	\end{aligned}
%\end{equation}
%    In particular, if we set $E^c\vert_{t=0} = E_0$, then there exists a constant $C>0$ not depending on $c$ such that
%\begin{equation}
%	\begin{aligned}
%		\| (u^c,\, E^c,\, B^c) - (u,\,E,\,B) \|_{L^\infty((0,T] ; L^2(\bbR^3))} \leq Cc^{-2}.
%	\end{aligned}
%\end{equation}
%}
%\end{remark}

\subsection{Outline of the paper}
This paper is organized as follow. In section \ref{sec2}, we investigate the local existence and uniqueness of strong solutions to the MHD system and the Euler-Maxwell system, which validate the assumptions of Theorem~\ref{thm_conv}. Using this, we improve convergence sense to $L^p_tH^s_x$. In section \ref{Auxiliary_linear_system}, we introduce an auxiliary linear system and we conduct a convergence analysis of the linear part using the energy estimate. We also obtain the convergence rate of linear parts. In section \ref{finalproof}, we derive the error part estimate using $L^2$-energy estimate. We combine these results to finish the main theorem and corollaries.

\section{Local existence of solutions}\label{sec2}

\subsection{Existence of local-in-time solutions}

In this section, we study the local time existence and uniqueness of the Euler-Maxwell system and the MHD system. Let us first consider the Euler-Maxwell equations. We will show that the local in time existence theory holds for each $c \geq c_0$ for some fixed $c_0 > 0$. Additionally, its existence time $T^c$ defined for each $c\geq c_0$ has a uniform lower bound. These results are necessary for our framework of limit problem, and validate the assumptions of Theorem~\ref{thm_conv}. 
\begin{theorem}[Euler-Maxwell system] \label{thm_ext}
Consider the initial data $(u_0,\,E_0,\,B_0) \in H^{m}(\bbR^3)$ for some integer $m > \frac {5}{2}$ with $\operatorname{div} u_0 =  \operatorname{div}{E_0} =\operatorname{div}{B_0} =0$. Then, there exist constants $M>0$ and $T > 0$ depending on $m$, $\|u_0\|_{H^m}$, $\| E_0 \|_{H^m}$, and $\|B_0\|_{H^m}$ such that the Euler-Maxwell system \eqref{main_eqn} possesses a unique solution
	\begin{equation}\label{un_1916}
	\begin{aligned}
		(u,\, E,\, B) \in C([0,T] ; H^m(\bbR^3)), \qquad j \in L^2([0,T] ; H^m(\bbR^3)),
	\end{aligned}
\end{equation}
which satisfies
\begin{equation}\label{uniform_esti_2323}
\begin{aligned}
			\sup_{t \in [0,T]} \left(\| u(t) \|_{H^m}^2 + \| E(t) \|_{H^m}^2 + \| B(t) \|_{H^m}^2\right)^{\frac{1}{2}}  + \left(\int_0^T \|j(t)\|_{H^m}^2\,\ud t\right)^{\frac{1}{2}}  \leq M.
\end{aligned}
\end{equation}
%For  $T^*>0$, it holds that
%\begin{equation}\label{claim_1020}
%    \sup_{t \in [0,T^*]}\left(\|u(t)\|_{H^m}^2 + \frac{1}{c^2} \|E(t)\|^2_{H^m} + \|B(t)\|^2_{H^m}\right) + \int_0^{T^*} \|j(t)\|^2_{H^m}\,\dt < +\infty
%\end{equation}
%if and only if 
%\begin{equation}\label{blow-up}
%    \int_0^{T^*} \| \nabla u (t) \|_{L^{\infty}}  + \| u (t) \|_{L^{\infty}}^2 +  \| B(t) \|_{L^{\infty}}^2  + \| j (t) \|_{L^{\infty}} \, \dt <+\infty.
%\end{equation}

\end{theorem}

We first state one lemma which demonstrate an a priori estimate for Euler-Maxwell system. Below we set

\begin{equation}\label{X_t}
    X(t) := \|u(t)\|_{H^m}^2 + \|E(t)\|^2_{H^m} + \|B(t)\|^2_{H^m},    
\end{equation}
and
\begin{equation}\label{A_t}
    A(t) : =  \| \nabla u (t) \|_{L^{\infty}}  + \| u (t) \|_{L^{\infty}}^2 + \| B(t) \|_{L^{\infty}}^2  + \| j (t) \|_{L^{\infty}}.
\end{equation}

\begin{lemma} [a priori estimate for Euler-Maxwell] \label{Energy_estimate_1015}
Let $m$ be an integer satisfying $m > \frac {5}{2}$. Assume that $(u,E,B) \in C([0,T);H^{m}(\mathbb{R}^3)) \cap \operatorname{Lip}(0,T;H^{m-1}(\bbR^3))$ satisfies \eqref{main_eqn}. Recall \eqref{X_t}-\eqref{A_t}.
There exists a constant $C=C(m)>0$ such that
it follows 
\begin{equation}\label{stepk_est}
    \frac {1}{2} \frac {\ud}{\ud t} X(t) + \frac{1}{2}\| j(t) \|_{H^m}^2 
		\leq C  A(t) X(t).
\end{equation}
In particular, it holds that
\begin{equation}\label{eng_Hm}
    \begin{gathered}
        \frac {1}{2} \frac {\ud}{\ud t} X(t) + \frac{1}{4}\| j \|_{H^m}^2 \\
		\leq C (\| \nabla u \|_{L^{\infty}} + \| u \|_{L^{\infty}}^2 +  \| B \|_{L^{\infty}}^2  ) X(t)
        + CX(t)^2\\
        \leq C (1+X(t))^2.
    \end{gathered}
\end{equation}
\end{lemma}
We provide the proof of Lemma \ref{Energy_estimate_1015} in the appendix.
\begin{proof}[Proof of Theorem~\ref{thm_ext}] It suffices to present a uniform a priori estimate for the solutions.
Let $c\geq c_0$ be fixed temporarily. We construct a sequence of solutions of the mollified
system \eqref{main_eqn} using standard existence theory (e.g. Theorem~3.1, \cite{majda2001vorticity}). Then, each solutions satisfy the a priori estimate \eqref{eng_Hm} on the uniform time interval, and this allows us to pass to a converging subsequence and show the existence of smooth solutions on local-in times.
To see that the solutions exist on the uniform time $[0,T]$ on each $c \geq c_0$, we set
\begin{equation*}
     y(t) = y_c(t) := X(t) +1 = \| u(t) \|_{H^m}^2 + \| E(t) \|_{H^m}^2 + \| B(t) \|_{H^m}^2+1.
\end{equation*}
Then, it follows from \eqref{eng_Hm} that
\begin{equation}\label{ODE_sun_2042}
    \frac {\ud}{\ud t} y(t) + \| j(t) \|_{H^m}^2 \leq C\,y(t)^{2}.
\end{equation} We emphasize that the constant $C>0$ in \eqref{ODE_sun_2042} is independent to the parameter $c \geq c_0$, and it  depends only on $m>0$.  Thus, it holds
\begin{equation*}
    y(t)\leq \frac{y(0)}{1-Cy(0) t}.
\end{equation*} Note that \begin{equation*}
    y(0) \leq \| u_0 \|_{H^m}^2 +  \| E_0 \|_{H^m}^2 + \| B_0 \|_{H^m}^2+1.
\end{equation*} 
We take a positive number $T$ with
\begin{equation*}1-C \left( \| u_0 \|_{H^m}^2 + \| E_0 \|_{H^m}^2 + \| B_0 \|_{H^m}^2+1 \right) T = \frac{1}{4}.
\end{equation*} Then, on the time interval $[0,T]$, it holds \begin{equation}\label{y_est}
    y(t) \leq 2 \left( \| u_0 \|_{H^m}^2 + \| E_0 \|_{H^m}^2 + \| B_0 \|_{H^m}^2+1 \right),
\end{equation} where $T$ is uniformly bounded below with respect to $c \in [c_0,\infty)$. We left the proof of the uniqueness to the appendix since it is standard.
%This implies
%$$\| u(t) \|_{H^m}^2 + \frac 1{c^2} \| E(t) \|_{H^m}^2 + \| B(t) \|_{H^m}^2 = y(t) \leq \frac{y_0}{1-Cy_0t} = \frac{\| u_0 \|_{H^m}^2 + \frac 1{c^2} \| E_0 \|_{H^m}^2 + \| B_0 \|_{H^m}^2}{1-C(\| u_0 \|_{H^m}^2 + \frac 1{c^2} \| E_0 \|_{H^m}^2 + \| B_0 \|_{H^m}^2)t}$$
%for all
%$t < 1/C(\| u_0 \|_{H^m}^2 + \frac 1{c^2} \| E_0 \|_{H^m}^2 + \| B_0 \|_{H^m}^2)$.
%Take $T = 1/2C(\| u_0 \|_{H^m}^2 + \frac 1{c^2} \| E_0 \|_{H^m}^2 + \| B_0 \|_{H^m}^2)$.
%To prove the blow-up criterion \eqref{blow-up}, we combine \eqref{blow-up} together with Gr\"onwall's inequality to conclude.
\end{proof}

It is well-known that the MHD system is locally well-posed, which is well established in \cite{chae2014well}.  
\begin{theorem}[MHD system] \label{thm_ext_inviscid_MHD}
Let $m$ be an integer satisfying $m > \frac {5}{2}$. Consider the initial data $(u_0,\,B_0) \in H^{m}(\bbR^3)$ with $\operatorname{div} u_0  =\operatorname{div}{B_0} =0$. Then there exist constants $T > 0$ and $\overline{M}>0$ depending on $m$, $\| u_0 \|_{H^m}$ and $\| B_0 \|_{H^m}$ such that the MHD system \eqref{eq-MHD} possesses a unique solution
	\begin{equation}\label{ass_eng_MHD_0}
	\begin{aligned}
		(\overline{u},\, \overline{B}) \in C([0,T] ; H^m(\bbR^3)), \qquad \nabla \overline{B} \in L^2([0,T] ; H^m(\bbR^3)),
	\end{aligned}
\end{equation}
which satisfies
	\begin{equation}\label{ass_eng_MHD}
		\begin{aligned}
			\sup_{t \in [0,T]} \left(\| \overline{u}(t) \|_{H^m}^2  + \| \overline{B}(t) \|_{H^m}^2\right)^{\frac{1}{2}}  + \left(\int_0^{T} \|\nabla \overline{B}(t)\|_{H^m}^2 \,\ud t\right)^{\frac{1}{2}} \leq \overline{M}.
		\end{aligned}
	\end{equation}
\end{theorem}
\begin{proof}
Notice that smooth solutions to MHD system \eqref{eq-MHD} satisfy the formal energy conservation law
\begin{equation}\label{L2_est_MHD}
    \frac12 \frac{\ud}{\ud t} \left(  \| u \|_{L^2}^2+\| B \|_{L^2}^2 \right) + \| \nabla B \|_{L^2}^2 = 0,
\end{equation} and the following energy estimates
\[
\frac12 \frac{\ud}{\ud t} \left(  \| u \|_{H^m}^2+\| B \|_{H^m}^2 \right) + \| \nabla B \|_{H^m}^2 \lesssim (\| \nabla u \|_{L^{\infty}} + \| \nabla B \|_{L^{\infty}}) \left( \| u \|_{H^m}^2 + \| B \|_{H^m}^2\right), \qquad m > \frac 52.
\]
See \cite[Theorem~2.2]{chae2014well} for the details.
\end{proof}
\begin{remark}
 The estimates \eqref{uniform_esti_2323} and \eqref{ass_eng_MHD} validate the assumptions of Theorem~\ref{thm_conv}. We also emphasize that $c_0>0$ can be taken as an arbitrary small positive constant.
\end{remark}

\subsection{Improvement of the convergence}
We recall that Theorem~\ref{thm_conv} is based on the assumptional solution satisfying \eqref{ass_eng_EM} (or additionally \eqref{ass_eng_EM_2}). We again emphasize that \eqref{ass_eng_EM} and \eqref{ass_eng_EM_2} does not implies the uniform boundedness of other norms, such as $\|u^c\|_{L^{\infty}_tH^m_x}$.
We remark that the condition \eqref{uniform_esti_2323} implies the assumption \eqref{ass_eng_EM} and \eqref{ass_eng_EM_2}. Thus, on the time interval $[0,T]$ satisfying \eqref{uniform_esti_2323}, it is clear that the convergence results in Corollary~\ref{cor_conv_uB} hold. 

Combining with the Gagliardo-Nirenberg interpolation inequality and Corollary~\ref{cor_conv_uB}, one finds that the sense of convergence can be improved on the time interval $[0,T]$, since we can use the uniform estimate \eqref{uniform_esti_2323}. We state this without the proof as follows.

\begin{corollary}[Improvement of Convergence]
Let $(u^c,B^c,E^c)$ be a solution of Euler-Maxwell system \eqref{main_eqn} satisfying Condition \ref{cond_initial}. Let $(\overline{u},\overline{B})$ be a solution of MHD system \eqref{eq-MHD} with initial data $(u_0,B_0)$. Let $0<s<m$. Then, there exists a time $T$\footnote{Indeed, we can choose $T>0$ satisfying \eqref{ass_eng_MHD_0}-\eqref{ass_eng_MHD} and \eqref{un_1916}-\eqref{uniform_esti_2323}} such that: 
\vspace{0.5em}
\newline
$(u^c,B^c)$ converges to $(\overline{u},\,\overline{B})$ in $L^{\infty}([0,T];H^s(\mathbb{R}^3))$, in the sense that
\begin{equation*}
    \lim_{c \to \infty}\|(u^c -\overline{u}, B^c - \overline{B})\|_{L^{\infty}(0,T;H^s(\mathbb{R}^3))}=0.
\end{equation*}
The electric fields $E^c$ converges to zero in $L^p(0,T;H^s(\mathbb{R}^3))$ for every $p\in[1,\infty)$, i.e.
\begin{equation*}
    \lim_{c \to \infty}\|E^c\|_{L^{p}(0,T;H^s(\mathbb{R}^3))}=0.
\end{equation*}
Also, the electric currunt $j^c$ converges to $\overline{j}$ in $L^p(0,T;H^s(\mathbb{R}^3))$ for every $p \in [1,2)$, i.e. 
\begin{equation*}
    \lim_{c \to \infty}\|j^c-\overline{j}\|_{L^{p}(0,T;H^s(\mathbb{R}^3))}=0.
\end{equation*}
If we further assume that $\|E^c_0\|_{L^2(\mathbb{R}^3)} \to 0$, then
$E^c$ converges to zero in $L^{\infty}(0,T;H^s(\mathbb{R}^3))$ and $j^c$ converges to $\overline{j}$ in $L^{2}(0,T;H^s(\mathbb{R}^3))$.
\end{corollary}

\section{Auxiliary linear system}\label{Auxiliary_linear_system} 
\begin{Notation} We use the following notation in this section.
\begin{itemize}
\item For any positive $A$ and $B$, we use the notation $A\lesssim B$ to mean that there exists a generic constant $C > 0$ such that $A \leq CB$. Moreover, $A \lesssim_{D_1,..,D_k}B$ is denoted when $A \leq C B$ and the constant $C$ depends on $D_1, ..., D_k$.
\end{itemize}
\end{Notation}

\subsection{Preliminaries}
Let $m > \frac{5}{2}$ and $(\overline{u},\, \overline{B})$ be a smooth solution to MHD system \eqref{eq-MHD} on some given time interval $[0,T]$ such that \eqref{ass_eng_MHD_0}-\eqref{ass_eng_MHD} holds. We recall the definition of $\overline{E}$ in \eqref{def_Ebar}. Given $\overline{E}$ and $E_0^c$, we consider the \textit{auxiliary linear system} associated $\eqref{main_eqn}_{2-3}$,
\begin{equation}\label{linear_H}
	\begin{cases}
	    \begin{aligned}
		&\frac{1}{c}\partial_t E_{L} - \nabla \times B_{L} + cE_{L} = \overline{E},\quad &\operatorname{div} E_L =0, \\
		&\frac{1}{c}\partial_t B_{L} + \nabla \times E_{L} = \frac{1}{c} \nabla \times \overline{E}, \quad &\operatorname{div} B_L = 0,\\
        &\qquad E_{L}\vert_{t=0}=E^c_0,\quad B\vert_{t=0}=0.
	\end{aligned}
	\end{cases}
\end{equation}
Note that the system \eqref{linear_H} depends on $c\geq c_0$ and thus so does its solution, but we denote the solution by $E_L$, $B_L$ rather $E_L^c$, $B_L^c$ for the readability. Since $E_0^c$ and $\overline{E}$ are divergence-free, the divergence-free condition is preserved over time. We claim that the solutions $cE_L$ and $B_L$ in \eqref{linear_H} converge to $\overline{E}$ and $0$ respectively and this observation is the key ingredient of our main theorem. 

To study \eqref{linear_H}, we should first discuss the existence of solution to such an linear system exists for each fixed $c\geq c_0$. For this, we let $\mathbf{B} := B_L + \overline{B}$ and consider an equivalent system to \eqref{linear_H}:
\begin{equation}\label{Linear_system}
	\begin{cases}
	    \begin{aligned}
		&\frac{1}{c}\partial_t E_{L} - \nabla \times \mathbf{B} + cE_{L} = \mathbb{P} (\overline{u} \times \overline{B}),\qquad &\operatorname{div} E_L =0, \\
		&\frac{1}{c}\partial_t \mathbf{B} + \nabla \times E_{L} = 0, \qquad &\operatorname{div} \mathbf{B} = 0,\\
        &E_{L}\vert_{t=0}=E^c_0,\quad \mathbf{B}\vert_{t=0}=\overline{B}\vert_{t=0}.
	\end{aligned}
	\end{cases}
\end{equation}
Note that
\begin{equation}\label{E_space}
    \overline{E} -\nabla \times \overline{B} = \mathbb{P} (\overline{u} \times \overline{B}) \in C([0,T];H^{m}(\mathbb{R}^3)).
\end{equation}
%Note from \eqref{E_space} that \begin{equation*}
%  \| \mathbb{P} (\overline{u} \times \overline{B})\|_{L^2(0,T;H^m(\bbR^3))} < \infty.
%\end{equation*} 

We provide a proposition regarding the existence of local-in-time solutions of \eqref{Linear_system}, thus of \eqref{linear_H}. Since one can prove it in the standard way, we omit the proof.
\begin{proposition} Let $m > \frac{5}{2}$ and $c \geq c_0$ be fixed. For any initial data $\phi_1 ,\, \phi_2 \in H^{m}(\bbR^3)$ with $\operatorname{div} \phi_1=\operatorname{div} \phi_2=0$, and any source term $f \in C([0,T];H^{m}(\mathbb{R}^3))$, there exists a unique solution $(E_L,\, \mathbf{B}) \in C([0,T] ; H^{m}(\bbR^3)) \cap C^1((0,T) ; H^{m-1}(\bbR^3))$ to the linear system
\begin{equation*}
	\begin{cases}
	    \begin{aligned}
		&\frac{1}{c}\partial_t E_{L} - \nabla \times \mathbf{B} + cE_{L} = \mathbb{P}f,\qquad &\operatorname{div} E_L =0, \\
		&\frac{1}{c}\partial_t \mathbf{B} + \nabla \times E_{L} = 0, \qquad &\operatorname{div} \mathbf{B} = 0,\\
        &E_{L}\vert_{t=0}= \phi_1,\quad \mathbf{B}\vert_{t=0}= \phi_2,
	\end{aligned}
	\end{cases}
\end{equation*} with the a priori bound: \begin{equation}\label{bdd_linear}
\begin{gathered}
    \sup_{t \in [0,T]} \left(\| E_L(t) \|_{H^m}^2 + \| B_L(t) \|_{H^m}^2 \right) + \int_0^T \| cE_L(t) \|_{H^m}^2 \,\ud t \leq \| \phi_1 \|_{H^m}^2 + \| \phi_2 \|_{H^m}^2 + \int_0^T \| f(t) \|_{H^m}^2 \,\ud t.
\end{gathered}
\end{equation}
\end{proposition}

Next, we introduce useful lemma which presents the estimate of $\partial_t\overline{E}$.

\begin{lemma}\label{lem_useful}
    For $m \geq 3$, let $(\overline{u},\,\overline{B})$ be a solution of MHD system \eqref{eq-MHD} satisfying \eqref{ass_eng_EM}-\eqref{ass_eng_MHD2016}, and we consider $\overline{E} \in C([0,T];H^{m-1}(\mathbb{R}^3))\cap L^2([0,T];H^m(\mathbb{R}^3))$ defined by \eqref{def_Ebar}. 
    Then, there exists a constant $C>0$ depending on $m$ and $\overline{M}$ such that
\begin{equation}\label{rdt_E}
        \| \partial_t \overline{E} \|_{L^{\infty}(0,T;H^{m-3})} + \| \partial_t \overline{E} \|_{L^2(0,T;H^{m-2})} \leq C.
    \end{equation}
In particular, there exists a constant $\tilde{C}$, depending only on $\overline{M}$ such that
\begin{equation}\label{rdt_E_110}
     \| \partial_t \overline{E} \|_{L^{\infty}(0,T;L^2)} + \| \partial_t \overline{E} \|_{L^2(0,T;H^1)} \leq \tilde{C}.
\end{equation}
\end{lemma}
\begin{proof}
    We estimate $\| \partial_t \overline{E} \|_{L^{\infty}(0,T;H^{m-3})}$ first. We recall \eqref{def_Ebar} and have
\begin{equation*}
\begin{aligned}
        \| \partial_t \overline{E} \|_{H^{m-3}} &= \| \partial_t \bbP(\nabla \times \overline{B}-\overline{u} \times \overline{B})  \|_{H^{m-3}} \\
        &\leq \| \nabla \partial_t \overline{B} \|_{H^{m-3}} + \|  \partial_t \overline{u} \|_{\dot{H}^{m-3}} \| \overline{B} \|_{L^{\infty}} + \| \overline{u} \|_{L^{\infty}} \| \partial_t \overline{B} \|_{\dot{H}^{m-3}}.
\end{aligned}
\end{equation*}
From the first equation of \eqref{eq-MHD}, it follows \begin{equation}\label{ubar_est}
     \|  \partial_t \overline{u} \|_{\dot{H}^{m-3}} \leq  \| \partial_t \overline{u} \|_{H^{m-1}} \leq \| (\overline{u} \cdot \nabla) \overline{u} \|_{H^{m-1}} + \| (\nabla \times \overline{B}) \times \overline{B} \|_{H^{m-1}} \leq 2 \overline{M}^2.
\end{equation} We have from the second equation of \eqref{eq-MHD} \begin{equation}\label{Bbar_est1}
    \| \partial_t \overline{B} \|_{\dot{H}^{m-3}} \leq \| \partial_t \overline{B} \|_{H^{m-2}} \leq \| \Delta \overline{B} \|_{H^{m-2}} + \| \overline{u} \times \overline{B} \|_{H^{m-1}}  \leq \overline{M} + \overline{M}^2 .
\end{equation}
Thus, we obtain \begin{equation}\label{def_C1}
    \| \partial_t \overline{E} \|_{L^\infty(0,T;H^{m-3})} \leq  \overline{M} +2\overline{M}^2+ 3\overline{M}^3. 
\end{equation}
    
Next, we estimate $\| \partial_t \overline{E} \|_{L^2(0,T;H^{m-2})}$. It follows from \eqref{def_Ebar} that
\begin{equation}\label{rdtE_1}
\begin{aligned}
    \| \partial_t \overline{E} \|_{L^2(0,T;H^{m-2})} \leq &\| \nabla \partial_t \overline{B} \|_{L^2(0,T;H^{m-2})} + \| \partial_t \overline{u} \times \overline{B} \|_{L^2(0,T;H^{m-2})} + \| \overline{u} \times \partial_t \overline{B} \|_{L^2(0,T;H^{m-2})}.
\end{aligned}
\end{equation}
Note that there exists $\theta \in (0,1)$ such that \begin{equation*}
    \|\overline{B}\|_{L^{\infty}} \leq C(m)\|\nabla \overline{B}\|_{L^2}^{\theta}\|\nabla^{m+1} \overline{B}\|_{L^2}^{1-\theta},
\end{equation*}
and this leads to
\begin{equation}\label{eq1342tue}
    \|\overline{B}\|_{L^2(0,T;L^{\infty})} \leq C \, \|\nabla \overline{B}\|_{L^2(0,T;H^m)} \leq C\overline{M}.
\end{equation}
Using \eqref{ubar_est} and \eqref{eq1342tue}, we can see that 
\begin{equation*}\label{rdtE_1}
\begin{aligned}
    \| \partial_t \overline{u} \times \overline{B} \|_{H^{m-2}}&\lesssim \|\nabla^{m-2} (\partial_t \overline{u} \times \overline{B}) \|_{L^2} + \| \partial_t \overline{u} \times \overline{B} \|_{L^2}\\ 
    &\lesssim \| \nabla ^{m-2} \partial_t \overline{u} \|_{L^2} \| \overline{B} \|_{L^{\infty}} + \|  \partial_t \overline{u} \|_{L^{\infty}} \| \nabla ^{m-2}\overline{B} \|_{L^{2}} + \|  \partial_t \overline{u} \|_{L^2} \| \overline{B} \|_{L^{\infty}} 
\end{aligned}
    \end{equation*}
and
\begin{equation*}\label{rdtE_2}
    \| \partial_t \overline{u} \times \overline{B} \|_{L^2(0,T;H^{m-2})} \lesssim  \|  \partial_t \overline{u} \|_{L^{\infty}(0,T;H^{m})} \| \nabla \overline{B} \|_{L^2(0,T;H^{m})} \lesssim \overline{M}^3
\end{equation*}
Similarly, it follows from \eqref{Bbar_est1} that \begin{equation*}\label{rdtE_3}
    \begin{aligned}
        \| \overline{u} \times \partial_t \overline{B} \|_{L^2(0,T;H^{m-2})} &\leq \| \| \overline{u} \|_{H^{m-2}} \| \partial_t \overline{B} \|_{L^\infty} \|_{L^2(0,T)} + \| \| \overline{u} \|_{L^\infty}\| \partial_t \overline{B} \|_{H^{m-2}} \|_{L^2(0,T)} \\
        &\leq C \|  \overline{u} \|_{L^{\infty}(0,T;H^{m-1})} \| \partial_t \overline{B} \|_{L^2(0,T;H^{m-1})}\\
        &\leq C \overline{M} \| \partial_t \overline{B} \|_{L^2(0,T;H^{m-1})}.
    \end{aligned}
    \end{equation*} On the other hand, \begin{align*}
        \| \partial_t \overline{B} \|_{L^2(0,T;H^{m-1})} &\leq \| \Delta \overline{B} \|_{L^2(0,T;H^{m-1})} + \| \overline{u} \times \overline{B} \|_{L^2(0,T;H^{m})} \\
        &\leq \|\nabla \overline{B}\|_{L^2(0,T;H^m)} +  \|  \overline{u} \|_{L^{\infty}(0,T;H^{m})} \| \partial_t \overline{B} \|_{L^2(0,T;H^{m})}\\
        & \leq \overline{M} + \overline{M}^2.
    \end{align*}
Combining these, we have
\begin{equation*}
    \| \partial_t \overline{E} \|_{L^2(0,T;H^{m-2})} \lesssim_m \overline{M} + 2\overline{M}^2 + 2\overline{M}^3.
\end{equation*}
The proof is completed.
\end{proof}
    
%From \eqref{rdtE_1} to \eqref{rdtE_4} we have \begin{gather*}
%        \| \partial_t \overline{E} \|_{L^2(0,T;H^{m-2})} 
%        \leq C \overline{M}^{\frac{3}{2}}T^{\frac{1}{2}} + C(1 + \overline{M}^{\frac{1}{2}}) C_3.
%    \end{gather*} Therefore, \eqref{rdt_E} is obtained. This completes the proof.
%\end{proof} 

\subsection{Linear parts estimates}
We have shown that a solution $(E_L,B_L)$ to the auxiliary linear system \eqref{linear_H} exists for each $c\geq c_0$. We shall establish estimates that depend on $c$. This is a crucial step in our analysis, as we obtain detailed information during this process.
In the following proposition, we fix $m=3$ temporarily. Indeed, for the higher order derivatives, Proposition \ref{prop_Lin} is optimal in $c$ and this shall be discussed in Appendix.
\begin{proposition}[Estimates of the auxiliary linear system]\label{prop_Lin}
For $T>0$ and $p \in [1,\infty]$, let $(\overline{u},\, \overline{B})$ be a solution to MHD system \eqref{eq-MHD} satisfying \eqref{ass_eng_EM}-\eqref{ass_eng_MHD2016}. Recalling that $\overline{E}$ in \eqref{def_Ebar} satisfy \eqref{E_space}, let $(E_L,\, B_L)$ be the solution of auxiliary linear system \eqref{linear_H}. Then, the followings hold:
\begin{equation}\label{est_BL}
    \begin{gathered}
        \sup_{t \in [0,T]} \left(\int_{\mathbb{R}^3} |B_L(t)|^2 \,\ud x\right)^{\frac{1}{2}} + \left(\int_0^T\int_{\mathbb{R}^3} |\nabla B_L|^2 \,\ud x \ud t\right)^{\frac{1}{2}}  \lesssim_{T,\overline{M}} \, c^{-2} \|cE_0^c - \overline{E}(0)\|_{H^1} + c^{-2}
    \end{gathered}
    \end{equation} and \begin{equation}\label{est_EL}
        \left\| cE_L(t) - \overline{E}(t) - e^{-c^2t} (cE_0^c - \overline{E}(0)) \right\|_{L^p(0,T;L^2)} \lesssim_{T,\overline{M}} \, (1+c^{1-\frac{2}{p}})(c^{-2} \|cE_0^c - \overline{E}(0)\|_{H^1} + c^{-2})
    \end{equation} for all $c \geq c_0$.
\end{proposition}
\begin{remark}
In particular, for $p=1$, it follows from \eqref{est_EL} that
\begin{equation}\label{EL_L1L2_estimate}
\int_0^T\|cE_L(t)-\overline{E}(t)\|_{L^2}\,\dt \lesssim_{T,\overline{M}} \frac{1}{c^2} \left(\|cE_0^c - \overline{E}(0)\|_{H^1}+1\right).
\end{equation}
\end{remark}

\begin{proof}[Proof of Proposition~\ref{prop_Lin}]
    Let $(E_L,\,B_L) \in C([0,T];H^m(\bbR^3))$ be the solution to \eqref{linear_H}. Then, it holds \begin{equation}\label{diff_E_eq}
        \frac{1}{c^2}\partial_t (cE_{L} - \overline{E}) - \nabla \times B_{L} + (cE_L - \overline{E}) = -\frac{1}{c^2} \partial_t \overline{E}
    \end{equation} and \begin{equation}\label{diff_B_eq}
        \partial_t B_{L} + \nabla \times (cE_{L} - \overline{E}) = 0
    \end{equation} on time interval $[0,T]$. Let us briefly introduce the idea as follows. We formally let $c$ tend to infinity in \eqref{diff_E_eq} and substitute the result into \eqref{diff_B_eq}. Then, we obtain the heat equation $\partial_t B_L - \Delta B_L=0$ with $B_L(0)=0$ formally at $c = \infty$. Considering the nature of this limit equation, we derive estimates by constructing parabolic estimates for the auxiliary linear system \eqref{linear_H}.
 \medskip \\
{\bf{Step 1} (Parabolic estimate for $B_L$)}  We first claim that
\begin{equation*}\label{est_rough_3}
\begin{gathered}
    \frac{1}{c^2} \sup_{t \in [0,T]} \int |\nabla \times (cE_L - \overline{E})|^2 \,\ud x + \sup_{t \in [0,T]} \int |\nabla \times B_L|^2 \,\ud x + \int_0^T \int |\nabla \times (cE_L - \overline{E})|^2 \,\ud x \ud t \\
    \leq \frac{1}{c^4}\int_0^T \int |\nabla \partial_t \overline{E}|^2 \,\ud x \ud t + \frac{1}{c^2} \int |\nabla(c E_0^c - \overline{E}(0))|^2 \,\ud x.
\end{gathered}
\end{equation*} for all $c \geq c_0$. By testing $\nabla \times \nabla \times (cE_L-\overline{E})$ and $\nabla \times \nabla \times B_L$ to \eqref{diff_E_eq}-\eqref{diff_B_eq}, we see \begin{equation*}
\begin{gathered}
    \frac{1}{2} \frac{\ud}{\ud t} \left( \frac{1}{c^2}\int |\nabla \times (cE_L - \overline{E})|^2 \,\ud x + \int |\nabla \times B_L|^2 \,\ud x \right) + \int |\nabla \times (cE_L - \overline{E})|^2 \,\ud x \\
    = -\frac{1}{c^2} \int \nabla \times \partial_t \overline{E} \cdot \nabla \times (cE_L - \overline{E}) \,\ud x\\
 \leq \frac{1}{2} \int |\nabla \times (cE_L - \overline{E}) |^2 \,\ud x + \frac{1}{2c^4} \int | \nabla \partial_t \overline{E} |^2 \,\ud x.
\end{gathered}
\end{equation*}
Integrating over time and using the divergence free condition, it follows that
\begin{equation}\label{sun0818_1638}
    \begin{gathered}
    \sup_{t \in [0,T]}\frac{1}{c^2} \int |\nabla(cE_L - \overline{E})(t)|^2 \,\ud x + 
    \sup_{t \in [0,T]}\int |\nabla B_L(t)|^2 \,\ud x + \int_0^T \int |\nabla(cE_L(t) - \overline{E}(t))|^2 \,\ud x \ud t \\
    \leq \frac{1}{c^2} \int |\nabla(cE_L - \overline{E})(0)|^2 \,\ud x + \frac{1}{c^4}\int_0^T\int | \nabla \partial_t \overline{E} |^2 \,\ud x \ud t.
    \end{gathered}
\end{equation}
Moreover, it follows from \eqref{diff_E_eq} that
\begin{equation}\label{zerg1}
    -\frac{1}{c^2} \int \partial_t(cE_L - \overline{E}) \cdot \nabla \times B_L \,\ud x + \int |\nabla B_L|^2 \,\ud x - \int (cE_L - \overline{E}) \cdot \nabla \times B_L \,\ud x  = \frac{1}{c^2} \int \partial_t \overline{E} \cdot \nabla \times B_L \,\ud x.
\end{equation} On the other hand, it follows from \eqref{diff_B_eq} that \begin{equation}\label{zerg2}
    \frac{1}{2} \frac{\ud}{\ud t} \int |B_L|^2 \,\ud x + \int \nabla \times (cE_L - \overline{E}) \cdot B_L \,\ud x = 0,
\end{equation} \begin{equation}\label{zerg3}
    -\frac{1}{c^2} \int \nabla \times (\partial_t B_L) \cdot (cE_L - \overline{E}) \,\ud x - \frac{1}{c^2} \int |\nabla \times (cE_L - \overline{E})|^2 \,\ud x = 0.
\end{equation} Summing \eqref{zerg1}-\eqref{zerg3}, one has \begin{equation*}
    \begin{gathered}
        -\frac{1}{c^2} \frac{\ud}{\ud t}\int \nabla \times B_L \cdot (cE_L - \overline{E}) \,\ud x + \frac{1}{2} \frac{\ud}{\ud t} \int |B_L|^2 \,\ud x + \int |\nabla B_L|^2 \,\ud x \\
        = \frac{1}{c^2} \int \partial_t \overline{E} \cdot \nabla \times B_L \,\ud x + \frac{1}{c^2} \int |\nabla (cE_L - \overline{E})|^2 \,\ud x.
    \end{gathered}
\end{equation*} Integrating over time, we have 
%\begin{equation*}
%    \begin{gathered}
%        -\frac{1}{c^2} \int \nabla \times B_L(t) \cdot (E_L - \overline{E})(t) \,\ud x + \frac{1}{2} \int |B_L(t)|^2 \,\ud x + \int_0^t \int |\nabla B_L|^2 \,\ud x \ud t \\
%        = \frac{1}{c^2} \int_0^t \int \partial_t \overline{E} \cdot \nabla \times B_L \,\ud x \ud t + \frac{1}{c^2} \int_0^t \int |\nabla (E_L - \overline{E})|^2 \,\ud x \ud t.
%    \end{gathered}
%\end{equation*} Thus, 
\begin{equation*}
    \begin{gathered}
        \frac{1}{2} \int |B_L(t)|^2 \,\ud x + \int_0^t \int |\nabla B_L|^2 \,\ud x \ud t \\
        = \frac{1}{c^2} \int_0^t \int \partial_t \overline{E} \cdot \nabla \times B_L \,\ud x \ud t + \frac{1}{c^2} \int_0^t \int |\nabla (cE_L - \overline{E})|^2 \,\ud x \ud t + \frac{1}{c^2} \int  B_L(t) \cdot \nabla \times (cE_L - \overline{E})(t) \,\ud x.
    \end{gathered}
\end{equation*} 
By Young's inequality and Lemma \ref{lem_useful}, we deduce that
\begin{equation}\label{eq0121sun}
    \begin{aligned}
        &\frac{1}{4} \int |B_L(t)|^2 \,\ud x + \frac{1}{2} \int_0^t \int |\nabla B_L|^2 \,\ud x \ud t \\
        &\hphantom{\qquad} \leq \frac{1}{2c^4} \int_0^t \int |\partial_t \overline{E}|^2 \,\ud x \ud t + \frac{1}{c^2} \int_0^t \int |\nabla (cE_L - \overline{E})|^2 \,\ud x \ud t + \frac{1}{c^4} \int |\nabla(cE_L - \overline{E})(t)|^2 \,\ud x.
    \end{aligned}
\end{equation}
Combining \eqref{eq0121sun} and \eqref{sun0818_1638} with \eqref{rdt_E_110}, we have
\begin{equation}\label{H_est}
    \begin{gathered}
    \sup_{t \in [0,T]} \int |B_L|^2 \,\ud x + \int_0^T \int |\nabla B_L|^2 \,\ud x \ud t + \frac{1}{c^2}\sup_{t \in [0,T]} \int |\nabla B_L|^2 \,\ud x \\
%    \leq C \left( \frac{1}{c^4} \int_0^T \int |\partial_t \overline{E}|^2 \,\ud x \ud t + \frac{1}{c^2} \int_0^T \int |\nabla(E_L - \overline{E})|^2 \,\ud x \ud t + \frac{1}{c^4} \sup_{t \in [0,T]} \int |\nabla(E_L - \overline{E})|^2 \,\ud x \right).
    \leq 8 \left( \frac{1}{c^4} \int_0^T \int |\partial_t \overline{E}|^2 \,\ud x \ud t + \frac{1}{c^6}\int_0^T\int | \nabla \partial_t \overline{E} |^2 \,\ud x \ud t + \frac{1}{c^4} \int |\nabla(cE_0^c - \overline{E}(0))|^2 \,\ud x \right)\\
    \leq C\left(\frac{1}{c^4}+\frac{1}{c^4}\int |\nabla(cE_0^c - \overline{E}(0))|^2 \,\ud x \right).
\end{gathered}
\end{equation}
Here, the constant $C$ only depends on $\overline{M}$ due to \eqref{rdt_E_110}.
Now, we combine \eqref{sun0818_1638} together with \eqref{H_est} and employ \eqref{rdt_E_110} again to obtain \eqref{est_BL}.

{\bf{Step 2 (Estimates for $E_L - \overline{E}$)}} We show \eqref{est_EL}. By applying Duhamel's principle to the first equation of \eqref{linear_H}, we have \begin{equation*}
\begin{gathered}
    \left\| cE_L(t) - \overline{E}(t) - e^{-c^2t} (cE_0^c - \overline{E}(0)) \right\|_{L^2} \leq c^2\int_0^t e^{-c^2(t-\tau)} \| \nabla B_L(\tau) \|_{L^2} \,\ud \tau + \int_0^t e^{-c^2(t-\tau)}\|\partial_{t}\overline{E}(\tau)\|_{L^2} \,\ud \tau\\
    \leq \left((c^2e^{-c^2(\cdot)})*\| \nabla B_L(\cdot) \|_{L^2}\right)(t) +\left((e^{-c^2(\cdot)})*\| \partial_t \overline{E}(\cdot) \|_{L^2}\right)(t) 
\end{gathered}
\end{equation*}
Let $p \in [1,\infty]$. Then, taking $L^p$-norm over the time-interval $[0,T]$ yields \begin{equation}\label{N1N2}
\begin{gathered}
    \left\| cE_L(t) - \overline{E}(t) - e^{-c^2t} (cE_0^c - \overline{E}(0)) \right\|_{L^p(0,T;L^2)} \\
   \leq  \left\| (c^2e^{-c^2(\cdot)})*\| \nabla B_L(\cdot) \|_{L^2} \right\|_{L^p(0,T)} + \left\| (e^{-c^2(\cdot)})*\| \partial_t \overline{E}(\cdot) \|_{L^2} \right\|_{L^p(0,T)}\\
   =: N_1 + N_2.
    \end{gathered}
\end{equation}
To estimate $N_1$, we divide the case and we use Young's convolution inequality to obtain that
\begin{align*}
    N_1 & \leq
    \begin{cases}
        c^2 \| e^{-c^2t} \|_{L^1(0,T)} \| \nabla B_L \|_{L^p(0,T;L^2)}\quad (\text{if }1\leq p \leq2),\\
        c^2 \| e^{-c^2t} \|_{L^{\frac{2p}{p+2}}(0,T)} \| \nabla B_L \|_{L^2(0,T;L^2)}\quad (\text{if }2< p \leq \infty).\\
    \end{cases}
\end{align*}
In the case of $p=\infty$, we interpret the value of $\frac{2p}{p+2}$ as $2$. If $1 \leq p \leq 2$, then we have \begin{align*}
    c^2 \| e^{-c^2t} \|_{L^1(0,T)} \| \nabla B_L \|_{L^p(0,T;L^2)}&\leq \| \nabla B_L \|_{L^p(0,T;L^2)} \leq T^{\frac{1}{p}-\frac{1}{2}} \| \nabla B_L \|_{L^2(0,T;L^2)}.
%    \\ &\lesssim c^{-(2-\gamma)} \|\Psi_0\|_{H^1} + c^{-2}\\
\end{align*}
If $2<p \leq \infty$, it similarly follows that \begin{equation*}
    c^2 \| e^{-c^2t} \|_{L^{\frac{2p}{p+2}}(0,T)} \| \nabla B_L \|_{L^2(0,T;L^2)}\leq c^{{1-\frac{2}{p}}}\| \nabla B_L \|_{L^2(0,T;L^2)} .
    %\lesssim c^{-\frac{2}{p}-(1-\gamma)} \|\Psi_0\|_{H^1} + c^{-(1+\frac{2}{p})}.
\end{equation*}
Hence, we bound $N_1$ by \begin{equation}\label{est_up1}
    N_1 \leq T^{\frac{1}{p}-\frac{1}{2}} \| \nabla B_L \|_{L^2(0,T;L^2)} + c^{{1-\frac{2}{p}}}\| \nabla B_L \|_{L^2(0,T;L^2)} .
%    (c^{-(2-\gamma)} + c^{-\frac{2}{p}-(1-\gamma)}) \|\Psi_0\|_{H^1} + c^{-(1+\frac{2}{p})} + c^{-2}.
\end{equation}
On the other hand, we have the following upper bound $N_2$ \begin{equation}\label{est_up2}
   N_2 \leq \| e^{-c^2t}\|_{L^1(0,T)} \| \partial_t \overline{E} \|_{L^p(0,T;L^2)} \leq c^{-2} \| \partial_t \overline{E} \|_{L^p(0,T;L^2)}.
\end{equation}
From \eqref{N1N2}, \eqref{est_up1} and \eqref{est_up2}, one has \begin{equation*}\label{est_up}
\begin{gathered}
    \left\| cE_L(t) - \overline{E}(t) - e^{-c^2t} (cE_0^c - \overline{E}(0)) \right\|_{L^p(0,T;L^2)} \\
    \leq T^{\frac{1}{p}-\frac{1}{2}} \| \nabla B_L \|_{L^2(0,T;L^2)} + c^{{1-\frac{2}{p}}}\| \nabla B_L \|_{L^2(0,T;L^2)} + c^{-2} \| \partial_t \overline{E} \|_{L^p(0,T;L^2)}.
%    (c^{-(2-\gamma)} + c^{-\frac{2}{p}-(1-\gamma)}) \|\Psi_0\|_{H^1} + c^{-(1+\frac{2}{p})} + c^{-2}.
\end{gathered}
\end{equation*} By recalling \eqref{rdt_E_110} and \eqref{est_BL}, we conclude \eqref{est_EL}.
This completes the proof.
\end{proof}

\section{Proof of the main result}\label{finalproof}
 We recall the notation $\mathcal{E}^c_0$ defined in \eqref{Eczero_0214}, and we shall use the notation in this section.
\begin{equation*}
    \calE^c_0 := \|u^c_0-u_0\|_{L^2}+\|B^c_0-B_0\|_{L^2} + \frac{1}{c^2} \left(\| cE^c_0 - \overline{E}(0) \|_{H^1} +1\right).
\end{equation*}
\subsection{Error part estimates}
We assume all conditions for Theorem~\ref{thm_conv}. We recall the \textit{Error part}  $(\widetilde{u},\widetilde{E},\widetilde{B})$ in \eqref{Error_term}. Then, the error part $(\widetilde{u},\widetilde{E},\widetilde{B})$ is in $C([0,T];H^m(\mathbb{R}^3))$ and satisfy
\begin{equation}\label{Error_eqn_4}
	\begin{cases}
	    \begin{aligned}
		\partial_t \widetilde{u} + (\widetilde{u} \cdot \nabla)\widetilde{u} + \nabla \widetilde{p} &= j^c \times B^c - \overline{j} \times \overline{B} - (\overline{u} \cdot \nabla)\widetilde{u} - (\widetilde{u} \cdot \nabla)\overline{u}, &\operatorname{div} \widetilde{u}=0,\\
		\frac{1}{c}\partial_t \widetilde{E} -  \nabla \times \widetilde{B} +  c\widetilde{E} &= - \bbP (u^c \times B^c  - \overline{u} \times \overline{B}  ), &\operatorname{div} \widetilde{B}=0, \\
		\frac{1}{c}\partial_t \widetilde{B} + \nabla \times \widetilde{E} &= 0, &  \operatorname{div} \widetilde{E} = 0,\\
         \widetilde{u}\vert_{t=0}=u^c_0-u_0,\quad &\widetilde{E}\vert_{t=0}=0,\quad \widetilde{B}\vert_{t=0}= B^c_0-B_0.
	\end{aligned}
	\end{cases}
\end{equation} We provide an $L^2$-estimate for the error parts which holds uniformly in $c \in [c_0,\infty)$. This estimate, combined with Condition~\ref{cond_initial}, reveals that the error part always tend to zero as $c \to \infty$. This implies that the difference between MHD system and Euler-Maxwell system can be analyzed by solely investigating mainly the auxiliary linear system \eqref{linear_H}.

\begin{proposition}[Estimate of the error part]\label{tilde_2122}
	We assume all conditions for Theorem~\ref{thm_conv}. Let  $(\widetilde{u},\widetilde{E},\widetilde{B})$ be solution to  error part system~\eqref{Error_eqn_4}. Then, we have
\begin{equation}
\begin{aligned}\label{error_est}
\sup_{t \in [0,T]} \Big( \| \widetilde{u}(t) \|_{L^2} +  \| \widetilde{E}(t) \|_{L^2}  &+ \| \widetilde{B}(t) \|_{L^2} \Big) +c\left(\int_0^T \int_{\bbR^3} |\widetilde{E}(t) |^2 \,\ud x \,\ud t\right)^{\frac{1}{2}} \lesssim_{T,M,\overline{M}} \, \calE^c_0.
\end{aligned}
\end{equation}
\end{proposition}
\begin{proof}
Below, we set 
    \begin{equation*}
        \widetilde{X}(t) =\widetilde{X}_c(t):= \left( \| \widetilde{u}(t) \|_{L^2}^2 + \| \widetilde{E}(t) \|_{L^2}^2 + \| \widetilde{B}(t) \|_{L^2}^2 + \frac{1}{c^4}\|cE^c_0-\overline{E}(0)\|_{L^2}^2 + \frac{1}{c^4} \right)^{\frac{1}{2}}.
    \end{equation*}Then, we see that $\widetilde{X}(0) \simeq \calE^c_0$. We claim that it holds
    \begin{equation}\label{yc_ODE}
        \frac{\ud}{\dt}\left(\widetilde{X}^2\right) + c^2\| \widetilde{E} \|_{L^2}^2 \leq C\left( (1+\|B^c\|_{L^{\infty}}^2 + \|j^c\|_{L^{\infty}})\widetilde{X} + \|cE_L-\overline{E}\|_{L^2}\right) \widetilde{X}.
    \end{equation}
To see this, we write
\begin{align*}
    \|j^c-\overline{j}\|_{L^2} &= \|cE^c - \overline{E}\|_{L^2} + \|\mathbb{P}(u^c\times B^c - \overline{u}\times \overline{B})\|_{L^2}\\
    & \leq  \|cE_L - \overline{E}\|_{L^2} + \|c\widetilde{E}\|_{L^2} + \|\widetilde{u}\|_{L^2}\|B^c\|_{L^{\infty}} + \|\overline{u}\|_{L^{\infty}}\|B_L+\widetilde{B}\|_{L^2}.
\end{align*}
It follows from the $\widetilde{u}$-equation in \eqref{Error_eqn_4} that
\begin{align*}
    \frac {1}{2} \frac {\ud}{\ud t} \int |\widetilde{u}|^2 \,\ud x &=  \int j^c \times (B^c -\overline{B}) \cdot \widetilde{u} \,\ud x +\int (j^c - \overline{j}) \times \overline{B} \cdot \widetilde{u} \,\ud x  - \int (\widetilde{u} \cdot \nabla) \overline{u} \cdot \widetilde{u} \,\ud x.
\end{align*}
By Young's inequality, there exists a constant $C>0$ which depends only on $\overline{M}$ such that
\begin{align*}
    \bigg|\int (j^c - \overline{j}) &\times \overline{B} \cdot \widetilde{u} \,\ud x \bigg| \leq \|j^c-\overline{j}\|_{L^2}\|\overline{B}\|_{L^{\infty}}\|\widetilde{u}\|_{L^2} \\
    & \leq\left(\|cE_L - \overline{E}\|_{L^2} + (\|B^c\|_{L^{\infty}} +\|\overline{B}\|_{L^{\infty}})\|\widetilde{u}\|_{L^2}+ \|\overline{u}\|_{L^{\infty}}\|B_L+\widetilde{B}\|_{L^2}\right)\|\overline{B}\|_{L^{\infty}}\|\widetilde{u}\|_{L^2} + \frac{c^2}{4}\|\widetilde{E}\|^2_{L^2}\\
     & \leq C\left(\|cE_L - \overline{E}\|_{L^2}  + \|B_L\|_{L^2} + (1+\|B^c\|^2_{L^{\infty}})\widetilde{X}\right)\widetilde{X} + \frac{c^2}{4}\|\widetilde{E}\|^2_{L^2}.
\end{align*}
Here, we used $\|B^c\|_{L^{\infty}} \leq 1+\|B^c\|_{L^{\infty}}^2$ in the last inequality.
Moreover, there holds by \eqref{ass_eng_MHD2016} that
\begin{align*}
    \left| \int j^c \times (B^c - \overline{B}) \cdot \widetilde{u} \,\ud x \right| &\leq \|j^c\|_{L^{\infty}}(\|B_L\|_{L^2}+\|\widetilde{B}\|_{L^2})\|\widetilde{u}\|_{L^2}\leq \|j^c\|_{L^{\infty}}(\|B_L\|_{L^2}+\widetilde{X})\widetilde{X}, \\
    \left| \int (\widetilde{u} \cdot \nabla) \overline{u} \cdot \widetilde{u} \,\ud x \right|&\leq \|\nabla \overline{u}\|_{L^{\infty}}\|\widetilde{u}\|^2_{L^2} \leq \overline{M}\widetilde{X}^2. \\
\end{align*}
We combine these to deduce that
\begin{equation}\label{tilde_u_0227}
    \frac {1}{2} \frac {\ud}{\ud t} \int |\widetilde{u}|^2 \,\ud x \leq C_{\overline{M}}\left(\|cE_L - \overline{E}\|_{L^2}  + (1+\|j^c\|_{L^{\infty}})\|B_L\|_{L^2} + (1+\|B^c\|^2_{L^{\infty}}+\|j^c\|_{L^{\infty}})\widetilde{X}\right)\widetilde{X} + \frac{c^2}{4}\|\widetilde{E}\|^2_{L^2}
\end{equation}

On the other hand, it follows from the $\widetilde{E}$ and $\widetilde{B}$ equations in \eqref{Error_eqn_4} that
\begin{align*}
    &\frac {1}{2} \frac {\ud}{\ud t} \int (|\widetilde{E}|^2 + |\widetilde{B}|^2) \,\ud x +  c^2\int |\widetilde{E}|^2 \,\ud x = \int \mathbb{P}((u^c - \overline{u}) \times B^c) \cdot (c\widetilde{E}) \,\ud x +   \int \mathbb{P}(\overline{u} \times (B^c -\overline{B})) \cdot (c\widetilde{E}) \,\ud x. 
\end{align*}
It holds that
\begin{align*}
	\left|\int \mathbb{P}((u^c - \overline{u}) \times B^c) \cdot (c\widetilde{E}) \,\ud x\right| &\leq  \|\widetilde{u}\|_{L^2}\|B^c\|_{L^{\infty}}\|c\widetilde{E}\|_{L^2} 
    \leq 2\|\widetilde{u}\|_{L^2}^2\|B^c\|_{L^{\infty}}^2 + \frac{c^2}{8} \|\widetilde{E}\|_{L^2}^2\\
    & \leq 2\|B^c\|_{L^{\infty}}^2 \widetilde{X}^2 + \frac{c^2}{8} \|\widetilde{E}\|_{L^2}^2, 
    \\
    \left|\int \mathbb{P}(\overline{u} \times (B^c -\overline{B})) \cdot (c\widetilde{E}) \,\ud x\right| &\leq \|\overline{u}\|_{L^{\infty}}(\|B_L\|_{L^2}+\|\widetilde{B}\|_{L^2}) \|c\widetilde{E}\|_{L^2} \leq 2\|\overline{u}\|_{L^{\infty}}^2\left(  \|\widetilde{B}\|_{L^2}^2 + \|B_L\|_{L^2}^2 \right) + \frac{c^2}{8}\|\widetilde{E}\|_{L^2}^2\\
    & \leq 2\overline{M}^2\left(  \widetilde{X}^2 + \|B_L\|_{L^2}^2 \right) + \frac{c^2}{8}\|\widetilde{E}\|_{L^2}^2.
\end{align*}
This shows that
\begin{equation}\label{tilde_EB_0227}
    \frac {1}{2} \frac {\ud}{\ud t} \int (|\widetilde{E}|^2 + |\widetilde{B}|^2) \,\ud x +  c^2\int |\widetilde{E}|^2 \,\ud x \leq C_{\overline{M}}\left((1+\|B^c\|_{L^{\infty}}^2)\widetilde{X}^2 +\|B_L\|_{L^2}^2\right) + \frac{c^2}{4}\|\widetilde{E}\|_{L^2}^2.
\end{equation}
Note that \eqref{est_BL} shows that $\|B_L(t)\|_{L^2}\lesssim_{T,\overline{M}}\,\calE^c_0 \leq \,X(t)$, and we combine this to estimates \eqref{tilde_u_0227}-\eqref{tilde_EB_0227} to obtain the following $L^2$-energy inequality 
\begin{equation*}\label{conv_eng}
		\begin{aligned}
			\frac {\ud}{\ud t} \int ( |\widetilde{u}|^2 +|\widetilde{E}|^2 + |\widetilde{B}|^2 ) \,\ud x +  c^2\int |\widetilde{E}|^2 \,\ud x  \lesssim_{T,\overline{M}}  \left( \|cE_L-\overline{E}\|_{L^2} +(1+\|B^c\|_{L^{\infty}}^2 + \|j^c\|_{L^{\infty}})\widetilde{X}\right)\widetilde{X}.
		\end{aligned}
	\end{equation*} 
Thus, the claim \eqref{yc_ODE} is proved.
Now, dividing both sides of equation \eqref{yc_ODE} by $\widetilde{X}$, we have \begin{equation*}
        \frac{\ud}{\dt}\widetilde{X} \leq C_{T,\overline{M}}\left((1+\|B^c\|_{L^{\infty}}^2 + \|j^c\|_{L^{\infty}})\widetilde{X} + \|cE_L-\overline{E}\|_{L^2}\right).
    \end{equation*}
We set
\begin{equation*}
    f(t):= 1+\|B^c(t)\|_{L^{\infty}}^2 + \|j^c(t)\|_{L^{\infty}}
\end{equation*}
Recalling \eqref{EL_L1L2_estimate}, we use Gr\"onwall's inequality to obtain
\begin{align*}
    \sup_{t \in [0,T]} \widetilde{X}(t) &\leq C \int_0^T \left(e^{C\int_t^T f(\tau) \ud \tau}\cdot \| cE_L(t) - \overline{E}(t) \|_{L^2} \right) \ud t + e^{C\int_0^T f(t) \dt}\cdot\widetilde{X}(0) \\
    & \leq C e^{C(T+M)} \int_0^T \|cE_{L}(t) - \overline{E}(t)\|_{L^2} \, \dt + e^{C(T+M)} \widetilde{X}(0)\\ 
    &\lesssim_{T,M,\overline{M}} \,  \calE^c_0.
\end{align*}
On the other hand, one can conclude from \eqref{yc_ODE} that \begin{align*}
    c^2\int_0^T \int |\widetilde{E}(t) |^2 \,\ud x \,\ud t &\lesssim \widetilde{X}(0)^2 + (T+M) \sup_{t \in [0,T]} \widetilde{X}(t)^2 + \left(\sup_{t \in [0,T]} \widetilde{X}(t)\right) \int_0^T \| cE_L(t) - \overline{E}(t) \|_{L^2} \, \ud t\\ &\lesssim_{T,M,\overline{M}} \, (\calE^c_0)^2.
\end{align*} Therefore, \eqref{error_est} is obtained. This completes the proof.

\end{proof}
\subsection{Proof of the main results}\label{mainproof}
We conclude this paper by presenting the proofs of the main theorem and corollary.
\begin{proof}[Proof of Theorem~\ref{thm_conv}]
We show \eqref{c_rate1} first. From the solution decomposition and Proposition~\ref{tilde_2122}, we have \begin{equation*}
        \sup_{t \in [0,T]} \| u^c(t) - \overline{u}(t) \|_{L^2} \leq \sup_{t \in [0,T]} \| \widetilde{u}(t) \|_{L^2} \lesssim \, \calE^c_0,
    \end{equation*} for all $c \geq c_0$. Similarly, it follows from \eqref{est_BL} and Proposition~\ref{tilde_2122} that
    \begin{align*}
        \sup_{t \in [0,T]} \| B^c(t) - \overline{B}(t) \|_{L^2} &\leq \sup_{t \in [0,T]} \| B_L(t) \|_{L^2} + \sup_{t \in [0,T]} \| \widetilde{B}(t) \|_{L^2} \lesssim \, \calE^c_0,
    \end{align*}
for all $c \geq c_0$. Consequently, \eqref{c_rate1} is obtained.

To prove \eqref{c_rate2}, we fix $p \in [1,+\infty]$. We estimate $E^c$ first. We see that
\begin{equation*}
    \| cE^c(t) - (\overline{E}(t) + e^{-c^2t} (cE_0-\overline{E}(0))) \|_{L^2} \leq \| cE_L(t) - (\overline{E}(t) + e^{-c^2t} (cE_0-\overline{E}(0))) \|_{L^2} + \| c\widetilde{E}(t) \|_{L^2}.
\end{equation*} In the case of $p \leq 2$, we note from H\"older's inequality and \eqref{error_est} that \begin{equation*}
    \| c\widetilde{E} \|_{L^p(0,T;L^2)} \leq T^{\frac{1}{p} - \frac{1}{2}} \| c\widetilde{E} \|_{L^2(0,T;L^2)} \lesssim \, \calE^c_0.
\end{equation*} If $p \geq 2$, the interpolation inequality and \eqref{error_est} shows \begin{equation*}
    \| c\widetilde{E} \|_{L^p(0,T;L^2)} \leq \| c\widetilde{E} \|_{L^\infty(0,T;L^2)}^{1-\frac{2}{p}} \| c\widetilde{E} \|_{L^2(0,T;L^2)}^{\frac{2}{p}} \lesssim c^{1-\frac{2}{p}}\, \calE^c_0.
    \end{equation*} 
On the other hand, we recall \eqref{est_EL}
 \begin{align*}
    \| cE_L(t) - (\overline{E}(t) + e^{-c^2t} (cE^c_0-\overline{E}(0))) \|_{L^p(0,T;L^2)} &\lesssim  (1+c^{1-\frac{2}{p}})(c^{-2} \|cE_0^c - \overline{E}(0)\|_{H^1} + c^{-2}),
    \end{align*} and thus, we have
\begin{align*}
    \| cE^c(t) - &(\overline{E}(t) + e^{-c^2t} (cE^c_0-\overline{E}(0))) \|_{L^p(0,T;L^2)}\\
    & \leq  \| cE_L(t) - (\overline{E}(t) + e^{-c^2t} (cE^c_0-\overline{E}(0))) \|_{L^p(0,T;L^2)}+\|c\tilde{E}\|_{L^p(0,T;L^2)}\\
    &\lesssim (c^{1-\frac{2}{p}} + 1) \, \calE^c_0,
    \end{align*}
for all $c \geq c_0$.

To control the $j$ term, we see that
\begin{align*}
    \|u^c\times B^c - \overline{u}\times\overline{B}\|_{L^p(0,T;L^2)} &= \|(u^c-\overline{u})\times B^c + \overline{u}\times(B^c-\overline{B})\|_{L^p(0,T;L^2)} \\
    & \leq \|u^c-\overline u\|_{L^{\infty}(0,T;L^2)}\|B^c\|_{_{L^p(0,T;L^{\infty})}} + \|\overline{u}\|_{L^{\infty}(0,T;L^{\infty})}\|B^c-\overline{B}\|_{L^p(0,T:L^2)}.
\end{align*}
By the $j$ equation of \eqref{main_eqn}, it follows from \eqref{ass_eng_MHD2016}, \eqref{ass_eng_EM_2} that
\begin{align*}
    &\| j^c - (\overline{j} + e^{-c^2t} (cE^c_0-\overline{E}(0))) \|_{L^p(0,T;L^2)} \\
    &\hphantom{\qquad\qquad}\leq \| cE^c - (\overline{E} + e^{-c^2t} (cE^c_0-\overline{E}(0))) \|_{L^p(0,T;L^2)} + \|(u^c-\overline{u})\times B^c + \overline{u}\times(B^c-\overline{B})\|_{L^p(0,T;L^2)} \\
    &\hphantom{\qquad\qquad}\lesssim (c^{1-\frac{2}{p}} + 1)\, \calE^c_0.
    \end{align*} Hence, we obtain \eqref{c_rate2} and complete the proof.
\end{proof}

\begin{proof}[Proof of Corollary~\ref{cor_conv_uB}] 
We recall $\calE^c_0$ in \eqref{Eczero_0214} and note that  $\calE^c_0 \to 0$ as $c\to\infty$. 
Thus, for the convergence of $(u^c,B^c)$, we can easily obtain \eqref{limit_ub} from \eqref{c_rate1}.
For the electric field $E^c$, and for $p \in [1,\infty]$, we have
\begin{align*}
    \|E^c\|_{L^p(0,T;L^2)}\leq \frac{1}{c} &\| cE^c(t) - (\overline{E}(t) + e^{-c^2t} (cE_0-\overline{E}(0))) \|_{L^p(0,T;L^2)}\\ &+ \frac{1}{c}\|\overline{E}(t)\|_{L^p(0,T;L^2)} +  \frac{1}{c}\|e^{-c^2t}cE_0 \|_{L^p(0,T;L^2)} + \frac{1}{c}\|e^{-c^2t}\overline{E}(0)) \|_{L^p(0,T;L^2)}.
\end{align*}
It follows from \eqref{c_rate2} that
\begin{equation}\label{eq1125sat}
    \|E^c\|_{L^p(0,T;L^2)} \lesssim \left(\frac{1}{c}+\frac{1}{c^{\frac{2}{p}}}\right)\calE^c_0 + \frac{1}{c}+\frac{1}{c^{1+\frac{2}{p}}} + \frac{1}{c^{\frac{2}{p}}}\|E^c_0\|_{L^2(\mathbb{R}^3)}.
\end{equation}
This yields that \eqref{limit_E} for every finite $p \in [1,\infty)$.

In particular, for the case $p=\infty,$ the definition of $L^{\infty}_tL^2_x$-norm and \eqref{eq1125sat} lead to
\begin{equation}\label{eq2118mon}
      \|E^c_0\|_{L^2(\mathbb{R}^3)}\leq \|E^c\|_{L^{\infty}(0,T;L^2)} \lesssim \calE^c_0 + \frac{1}{c} + \|E^c_0\|_{L^2(\mathbb{R}^3)}.
\end{equation}
Clearly, this shows that $\|E^c\|_{L^{\infty}(0,T;L^2(\mathbb{R}^3))}$ converges to zero, when $\|E^c_0\|_{L^2}$ converges to $0$.
Moreover, if $\|E^c_0\|_{L^2}$ does not tend to zero, then $E^c$ does not converge any limit in $L^{\infty}_tL^2_x$. Since we have $E^c \to 0$ in the sense of $L^1_tL^2_x$, the only possible candidate of the limit of $E^c$ in $L^{\infty}_tL^2_x$ is also zero. This observation with \eqref{eq2118mon} finishes the proof.
\end{proof}

\begin{proof}[Proof of Corollary \ref{cor_conv_nonconv}] 
Since the analysis for the case of $j^c$ is similar, we only address the cases of $E^c$. 
Let $p \in [1,\infty)$. Using \eqref{c_rate2}, one can derive the estimate of $cE^c -\overline{E}$ by
\begin{equation}\label{upper_E}
\begin{aligned}
    \| cE^c - \overline{E}\|_{L^p(0,T;L^2)} &\leq \| cE^c - (\overline{E} + e^{-c^2t} (E_0^c-\overline{E}(0)) \|_{L^p(0,T;L^2)}  + \|e^{-c^2t} (cE_0^c-\overline{E}(0)) \|_{L^p(0,T;L^2)}\\ 
    &\lesssim (1+ c^{1-\frac{2}{p}})\,\calE^c_0 + c^{-\frac{2}{p}}\|cE_0^c-\overline{E}(0)\|_{L^2}.
\end{aligned}
\end{equation}
%Similar argument can be applied to the convergence for $\gamma=0$. 
Similarly, the lower bound estimate also holds by
\begin{equation}\label{lower_E}
\begin{aligned}
    \| cE^c - \overline{E}\|_{L^p(0,T;L^2)} &\geq  \|e^{-c^2t} (cE_0^c-\overline{E}(0)) \|_{L^p(0,T;L^2)}- \| cE^c - (\overline{E} + e^{-c^2t}) (E_0^c-\overline{E}(0)) \|_{L^p(0,T;L^2)}  \\ 
    &\gtrsim   c^{-\frac{2}{p}}\|cE_0^c-\overline{E}(0)\|_{L^2} -  (1+ c^{1-\frac{2}{p}})\,\calE^c_0.
\end{aligned}
    \end{equation}
Now, we can directly deduce all results of convergence  for $cE^c$ using \eqref{upper_E}.
To demonstrate the non-convergence, observe that we have $cE^c \to \overline{E}$ in $L^1_tL^2_x$. This implies that the only candidate of the limit of $cE^c$ in $L^p_tL^2_x$ is $\overline{E}$ for all $p \in [2,\infty]$. Combining this observation to \eqref{lower_E}, the we obtain the all results of non-convergence. The proof is finished.
\end{proof}

\begin{proof}[Proof of Corollary \ref{cor_energyjump}]
Let $t \in (0,T]$ be fixed. It follows from \eqref{energyjump0305} with replacing $E_0$ to $E^c_0$ that 
\begin{equation}\label{eq1206tue}
\begin{gathered}
    \left| \frac{1}{2} \| E^c(t) \|_{L^2}^2 - \frac{1}{2} \| E^c_0 \|_{L^2}^2 + \int_0^t\| (j^c-\overline{j})(\tau)\|_{L^2}^2 \,\ud \tau \, \right| \\
    \leq \frac{1}{2} \left(\|u^c(t)\|^2_{L^2}-\|\overline{u}(t)\|_{L^2}^2\right)+ \frac{1}{2}\left(\|B^c(t)\|^2_{L^2}-\|\overline{B}(t)\|_{L^2}^2\right)   + 2\left|\int_{\mathbb{R}^3}  \overline{j} \cdot (j^c - \overline{j}) \,\ud x \right|.
\end{gathered}
\end{equation}
Here, we used
\begin{equation*}
    \int_0^t\| (j^c-\overline{j})(\tau)\|_{L^2}^2 \,\ud \tau = \int_0^t\| j^c(\tau)\|_{L^2}^2 \,\ud \tau - 2\int_0^t (j^c-\overline {j}) \cdot \overline{j} \,\ud \tau - \int_0^t\| \overline{j}(\tau)\|_{L^2}^2 \,\ud \tau.
\end{equation*}
By the convergence result \eqref{limit_ub} of $(u^c,B^c)$ in Corollary~\ref{cor_conv_uB}, the first and second terms on the right-hand side in \eqref{eq1206tue} goes to $0$ as $c \to \infty$. Note that the uniform boundedness of $\|j^c\|_{L^2(0,T;L^2)}$ reveals $j^c$ weakly converges in $L^2_tL^2_x$ and also we have shown that $j^c \to \overline{j}$ in $L^1(0,T;L^2))$ in strong sense. Thus, $j^c$ converges to $\overline{j}$ weakly in $L^2_tL^2_x$. This reveals that the last term also converges to $0$. In other words, 
\begin{equation}\label{eq1133tue}
    \lim_{c\to\infty}\left( \left| \frac{1}{2} \| E^c(t) \|_{L^2}^2 - \frac{1}{2} \| E^c_0 \|_{L^2}^2 + \int_0^t\| (j^c-\overline{j})(\tau)\|_{L^2}^2 \,\ud \tau \, \right| \right)=0.
\end{equation}
On the other hands, \eqref{c_rate3} shows that
\begin{equation*}
    \left|\| j^{c} - \overline{j} \|_{L^2(0,t;L^2)} - \|e^{-c^2t}(cE^c_0-\overline{E}(0))\|_{L^2(0,t;L^2)}\right| \leq C \calE^c_0.
\end{equation*}
For each fixed $t>0$, we see that
\begin{equation*}
    \|e^{-c^2t}(cE^c_0-\overline{E}(0))\|_{L^2(0,t;L^2)} = \frac{1}{\sqrt2}\|E^c_0\|_{L^2}+o(1).
\end{equation*}
Thus, this shows that
\begin{equation}\label{eq1223tue}
\lim_{c\to \infty} \left( \int_0^t\| (j^c -\overline{j})(\tau) \|_{L^2}^2\,\ud \tau - \frac{1}{2}\|E_0^c\|_{L^2}^2 \right) = 0.
\end{equation}
Combining with \eqref{eq1133tue} and \eqref{eq1223tue}, we conclude that
\begin{equation*}
    \lim_{c \to \infty} \|  E^c(t) \|_{L^2}^2 = 0.
\end{equation*}
We combine this and \eqref{eq1721tue} to finish the proof.
\end{proof}

\section*{Acknowledgements}
The authors would like to thank Haroune Houamed for valuable comments on the draft.
\\
J. Kim's work was supported by the National Research Foundation of Korea(NRF) grant funded by the Korea government(MSIT) (No. RS-2024-00360798). J. Lee's work was supported by the National Research Foundation of Korea (NRF) grant funded by the Korea government (MSIT) (No. NRF-2021R1A2C1092830). 

\appendix
\section{Appendix}
\subsection{The a priori estimate}

We provide the proof of Lemma~\ref{Energy_estimate_1015}.

\begin{proof}[Proof of Lemma \ref{Energy_estimate_1015}] 
Let $\alpha$ be a multi-index with $|\alpha| = m$. Then, we have
\begin{equation}\label{est_Hk_eng1}
	\begin{gathered}
		\frac {1}{2} \frac {\ud}{\ud t} \int (|\partial^{\alpha} u|^2 + |\partial^{\alpha} E|^2 + |\partial^{\alpha} B|^2) \,\ud x \\
		= - \int \partial^{\alpha} (u \cdot \nabla)u \cdot \partial^{\alpha} u \,\ud x + \int \partial^{\alpha} (j \times B) \cdot \partial^{\alpha} u \,\ud x - \int \partial^{\alpha} j \cdot \partial^{\alpha} (cE) \,\ud x.
	\end{gathered}
\end{equation} The divergence free condition implies
\begin{equation}\label{est_Hk_eng2}
	\begin{aligned}
		\left| - \int \partial^{\alpha} (u \cdot \nabla)u \cdot \partial^{\alpha} u \,\ud x \right| &\leq C \| \nabla u \|_{L^{\infty}} \| \nabla^{m} u \|_{L^2} ^2.
	\end{aligned}
\end{equation} Since the $j$ equation gives
\begin{equation*}
	\begin{aligned}
		\partial^{\alpha} (cE) = \partial^{\alpha} j  - \partial^{\alpha} \mathbb{P}(u \times B)
	\end{aligned}
\end{equation*}
the remainder integrals are equal to
\begin{equation}\label{rem_est}
    \begin{gathered}
		\int \partial^{\alpha} (j \times B) \cdot \partial^{\alpha} u \,\ud x - \int \partial^{\alpha} j \cdot \partial^{\alpha} (cE) \,\ud x  \\
        = - \int |\partial^{\alpha} j|^2 \,\ud x + \int \partial^{\alpha} j \cdot \partial^{\alpha} (u \times B) \,\ud x + \int \partial^{\alpha} (j \times B) \cdot \partial^{\alpha} u \,\ud x.
    \end{gathered}
\end{equation}
The second and third terms of \eqref{rem_est} are bounded simply by the sum of
\begin{equation}\label{est_Hk_eng3}
    \begin{gathered}
        \left| \int \partial^{\alpha} j \cdot \partial^{\alpha} (u \times B) \,\ud x  \right| \\
        \leq  C \|\nabla ^m j\|_{L^2}\left( \| \nabla^{m} u \|_{L^2} \|  B \|_{L^{\infty}} + \| u \|_{L^{\infty}} \| \nabla^m B \|_{L^2} \right) \\
        \leq \frac{1}{8 \cdot 3^m} \| \nabla^m j \|_{L^2}^2 + C \| \nabla^m u \|_{L^2}^2  \| B \|_{L^{\infty}}^2 + C \| u \|_{L^{\infty}}^2 \| \nabla^m B \|_{L^2}^2 ,
    \end{gathered}
\end{equation} and
\begin{equation}\label{est_Hk_eng4} 
	\begin{gathered}
		\left|  \int \partial^{\alpha} (j \times B) \cdot \partial^{\alpha} u \,\ud x  \right| \\
        \leq C\| \nabla^{m} u \|_{L^2} \left( \| \nabla^{m} j \|_{L^2} \|  B \|_{L^{\infty}} + \| j \|_{L^\infty} \| \nabla^m B \|_{L^2} \right) \\
        \leq \frac{1}{8 \cdot 3^m} \| \nabla^m j \|_{L^2}^2 + C\| \nabla^m u \|_{L^2}^2  \|  B \|_{L^{\infty}}^2 + C\| j \|_{L^{\infty}} \| \nabla^m u \|_{L^2} \| \nabla^m B \|_{L^2}.
	\end{gathered}
\end{equation} 
Combining \eqref{est_Hk_eng1}-\eqref{est_Hk_eng4} together with $L^2$-energy estimate, it follows that \begin{equation}\label{blow_esti1215}
\begin{gathered}
    \frac {1}{2} \frac {\ud}{\ud t} (\| u \|_{H^m}^2 + \| E\|_{H^m}^2 + \| B \|_{H^m}^2) + \frac{1}{2}\| j \|_{H^m}^2 \\
		\leq C (\| \nabla u \|_{L^{\infty}} + \| u \|_{L^{\infty}}^2 +  \|  B \|_{L^{\infty}}^2 + \| j \|_{L^{\infty}} ) (\| u \|_{H^m} ^2 + \| B \|_{H^m}^2)
    \end{gathered}
\end{equation} Applying the following estimate \begin{equation*}
    \| j \|_{L^{\infty}} (\| u \|_{H^m} ^2 + \| B \|_{H^m}^2) \leq \frac{1}{8} \| j \|_{H^m}^2 + C(\| u \|_{H^m} ^2 + \| B \|_{H^m}^2)^2,
\end{equation*}
we obtain \eqref{eng_Hm}. Moreover, we combine the estimate \eqref{blow_esti1215} with the Gr\"onwall inequality to obtain the blow-up criterion. This completes the proof.
\end{proof}

\subsection{Uniqueness of the solution}
We next provide uniqueness result omitted in Theorem~\ref{thm_ext}. 
\begin{proof}[Proof of Theorem~\ref{thm_ext} (Uniqueness part)]
Let $(u_1,B_1,E_1)$ and $(u_2, B_2, E_2)$ be two solutions of \eqref{main_eqn} with the same initial data $(u_0,B_0,E_0)$.
Then, by subtracting one equation from another, we have
\begin{equation*}
\begin{cases}
	\begin{aligned}
		\partial_t (u_1-u_2) + (u_1 \cdot \nabla)(u_1-u_2) +((u_1-u_2) \cdot \nabla)u_2  + \nabla (p_1-p_2) &= j_1 \times B_1 - j_2 \times B_2, \\
		\frac {1}{c}\, \partial_t (E_1-E_2) - \nabla \times (B_1-B_2) &= -(j_1-j_2), \\
		\frac {1}{c}\,\partial_t (B_1-B_2) + \nabla \times (E_1-E_2) &= 0, \\
	\end{aligned}
\end{cases}
\end{equation*} 
and we easily verify that
\begin{equation}\label{ineq_1103}
\begin{aligned}
\frac 12 \frac {\ud}{\ud t} &\left( \| u_1-u_2  \|_{L^2}^2  + \| E_1-E_2 \|_{L^2}^2 + \| B_1-B_2 \|_{L^2}^2 \right) + \| j_1-j_2 \|_{L^2}^2 \\
\leq &\| u_1-u_2 \|_{L^2}^2 \| \nabla u_2 \|_{L^{\infty}} + \| j_1-j_2 \|_{L^2}\| u_1-u_2 \|_{L^2} \|B_1 \|_{L^{\infty}} \\
&+ 2\| j_2 \|_{L^{\infty}} \| B_1-B_2 \|_{L^2}\| u_1-u_2 \|_{L^2} + \| j_1-j_2 \|_{L^2}\| u_2 \|_{L^\infty} \|B_1-B_2 \|_{L^2}.
\end{aligned}    
\end{equation}
We denote $Y_c(t)$ by
$$Y_c(t) = \| u_1-u_2 \|_{L^2}^2 + \| E_1-E_2 \|_{L^2}^2 + \| B_1-B_2 \|_{L^2}^2.$$ 
It follows from \eqref{ineq_1103} that
$$\frac {\ud}{\ud t} Y_c(t) \leq C Y_c(t) \left(\| \nabla u_2 \|_{L^{\infty}} + \| B_1 \|_{L^{\infty}}^2 + \| j_2 \| + \| u_2 \|_{L^{\infty}}^2 \right),$$
which implies 
$$Y_c(t) \leq Y_c(0) \exp \left( C \int_0^t \left(\| \nabla u_2 \|_{L^{\infty}} + \| B_1 \|_{L^{\infty}}^2 + \| j_2 \| + \| u_2 \|_{L^{\infty}}^2 \right) \,\ud \tau \right).$$
Combining the boundedness of the solutions with $Y_c(0)=0$, we conclude that $Y_c(t)=0$ for $t \in [0,T]$. 
\end{proof}

\subsection{Optimality of the linear estimate}

We remark that the result of Propsition~\ref{prop_Lin} implies that the convergence results for $B_L$ and $E_L$ are sharp in the following sense: We assume that $E_0^c \neq 0$. Combining \eqref{est_EL} and
\begin{equation*} \| e^{-c^2t} (cE_0^c - \overline{E}(0)) \|_{L^p(0,T;L^2)} = \| e^{-c^2t} \|_{L^p(0,T)} \| cE_0^c - \overline{E}(0) \|_{L^2}\simeq c^{-\frac{2}{p}+1},
\end{equation*}
    there holds for $1 < p \leq \infty$ that 
    \begin{equation}\label{sharp_EL}
    \left\| cE_L - \overline{E} \right\|_{L^p(0,T;L^2)} \simeq c^{-\frac{2}{p}+1}
    \end{equation} for sufficiently large $c$. Moreover, if $m \geq 4$ is assumed further, then \begin{equation*}
        \| B_L \|_{L^{\infty}(0,T;H^{m-4})} + \| \nabla B_L \|_{L^2(0,T;H^{m-4})} \gtrsim c^{-1}
    \end{equation*} holds. To see this, we observe the following:
in the case of $m > 3$, the proposition~\ref{prop_Lin} can be extended to general one by considering the $(m-3)$-th order derivatives and making corresponding adjustments, thereby yielding the analogous conclusion for higher orders. In other words, we can replace \eqref{est_BL} and \eqref{est_EL} by \begin{equation}\label{ext_BL}
\begin{gathered}
    \sup_{t \in [0,T]} \left(\int_{\mathbb{R}^3} |\nabla^{m-3} B_L(t)|^2 \,\ud x\right)^{\frac{1}{2}} + \left(\int_0^T\int_{\mathbb{R}^3} |\nabla^{m-2} B_L|^2 \,\ud x \ud t\right)^{\frac{1}{2}} \lesssim_{m,T,\overline{M}} c^{-2} \|\nabla^{m-3}( c E_0^c - \overline{E}(0))\|_{H^{1}} + c^{-2}
\end{gathered}
\end{equation} and \begin{equation}\label{ext_EL}
   \left\| \nabla^{m-3}\left( cE_L(t) - \overline{E}(t) - e^{-c^2t} (cE_0^c - \overline{E}(0)) \right)\right\|_{L^p(0,T;L^2)} \lesssim_{m,T,\overline{M}} (1+c^{1-\frac{2}{p}})(c^{-2} \|\nabla^{m-3}( c E_0^c - \overline{E}(0))\|_{H^{1}} + c^{-2}),
\end{equation} respectively.
Using this, we obtain the following proposition which shows the optimality of the estimate \eqref{diff_B_eq}.
\begin{proposition}\label{prop_sharp}
    Assume all assumptions in Proposition~\ref{prop_Lin}. Furthermore, suppose that $E_0^c \neq 0$, and $m \geq 4$. Then, \begin{equation}\label{est_low_B}
        \left(\sup_{t\in[0,T]}\int_{\mathbb{R}^3} |\nabla^{m-4} B_L|^2 \,\ud x\right)^{\frac{1}{2}} + \left(\int_0^T \int_{\mathbb{R}^3} |\nabla^{m-3} B_L|^2 \,\ud x \ud t\right)^{\frac{1}{2}} \gtrsim c^{-1}
    \end{equation} holds for $c \geq c_0$.
\end{proposition}
\begin{proof}
For simplicity, we only prove \eqref{est_low_B} in case $m = 4$, since the other cases can be shown similarly. Let $(E_L,\,B_L) \in C([0,T];H^m(\bbR^3))$ be the solution to \eqref{linear_H}. By testing $\nabla \times (\overline{E} - cE_L)$ to the $B_L$ equation of \eqref{linear_1706}, we have \begin{equation*}
    \begin{aligned}
         \int_0^T \int |\nabla \times (\overline{E} - cE_L)|^2 \,\ud x \ud t &= \int_0^T \int \partial_t B_L \cdot \nabla \times (\overline{E} - cE_L) \,\ud x \ud t \\
         &= \int B_L(T) \cdot \nabla \times (\overline{E} - cE_L)(T) \,\ud x - \int_0^T\int B_L \cdot \nabla \times \partial_t (\overline{E} - cE_L) \,\ud x \ud t.
    \end{aligned}
\end{equation*}
Let us consider a small constant $\varepsilon>0$ which will be specified later. We have from Young's inequality
\begin{equation*}
    \left| \int B_L(T) \cdot \nabla \times (\overline{E} - cE_L)(T) \,\ud x \right| \leq \frac{c^2}{\varepsilon} \int |B_L(T)|^2 \,\ud x + \frac{\varepsilon}{4c^2} \int |\nabla \times (\overline{E} - cE_L)(T)|^2 \,\ud x,
\end{equation*} and similarly, \begin{equation*}
    \left| \int_0^T\int B_L \cdot \nabla \times \partial_t (\overline{E} - cE_L) \,\ud x \ud t \right| \leq \frac{c^2}{\varepsilon} \int_0^T \int |\nabla B_L|^2 \,\ud x \ud t + \frac{\varepsilon}{4c^2} \int_0^T \int |\partial_t (\overline{E} - cE_L)|^2 \,\ud x \ud t.
\end{equation*} Thus, it follows
\begin{equation*}
\begin{aligned}
         \int_0^T \int |\nabla \times (\overline{E} - cE_L)|^2 \,\ud x \ud t  &\leq \frac{c^2}{\varepsilon}\left( \int |B_L(T)|^2 \,\ud x + \int_0^T \int |\nabla B_L|^2 \,\ud x \ud t \right)\\ &\quad + \frac{\varepsilon}{4c^2} \left( \int |\nabla \times (\overline{E} - cE_L)(T)|^2 \,\ud x + \int_0^T \int | \partial_t (\overline{E} - cE_L)|^2 \,\ud x \ud t \right).
\end{aligned}
\end{equation*} 
We recall from \eqref{ext_EL} with $p = \infty$ that
\begin{equation*}
    \begin{aligned}
        \frac{1}{c^2} \int |\nabla \times (\overline{E} - cE_L)(T)|^2 \,\ud x \leq C.
    \end{aligned}
\end{equation*}
for some $C>0$. One can easily have from \eqref{diff_E_eq} that \begin{gather*}
    \frac{1}{c^2} \int_0^T \int |\partial_t (\overline{E} - cE_L)|^2 \,\ud x \ud t \\
    \leq 2 \left( \int_0^T \int |\nabla B_L|^2 \,\ud x \ud t + \int_0^T \int |cE_L - \overline{E}|^2 \,\ud x \ud t + \frac{1}{c^2} \int_0^T \int |\partial_t \overline{E}|^2 \,\ud x \ud t \right).
\end{gather*} Combining with \eqref{est_BL}, \eqref{sharp_EL}, and \eqref{rdt_E}, we deduce that \begin{equation*}
    \frac{1}{c^2} \int_0^T \int |\partial_t (\overline{E} - cE_L)|^2 \,\ud x \ud t \leq C
\end{equation*} for some $C > 0$. From the above results, we obtain 
\begin{equation*}
         \int_0^T \int |\nabla \times (\overline{E} - cE_L)|^2 \,\ud x \ud t  \leq \frac{c^2}{\varepsilon}\left( \int |B_L(T)|^2 \,\ud x + \int_0^T \int |B_L|^2 \,\ud x \ud t \right) + \varepsilon C.
\end{equation*} Since \eqref{sharp_EL} with $p =2$ implies \begin{equation*}
    \int_0^T \int |\nabla \times (\overline{E} - cE_L)|^2 \,\ud x \ud t \geq \tilde{C},
\end{equation*}
for some $\tilde{C}>0$, it follows that
\begin{equation*}
         \tilde{C} \leq \frac{c^2}{\varepsilon}\left( \int |B_L(T)|^2 \,\ud x + \int_0^T \int |B_L|^2 \,\ud x \ud t \right) + \varepsilon C.
\end{equation*}
Note that the constants $C$ and $\tilde{C}$ do not depend on $c \geq c_0$. By taking $\varepsilon>0$ small so that $\varepsilon \leq \frac{\tilde{C}}{2C}$, we conclude \eqref{est_low_B}.
This completes the proof.
\end{proof}

%------------------------------------------------------------------------------%

\end{document}